\documentclass{amsart}

\usepackage{amssymb} \usepackage{amsfonts} \usepackage{amsmath}
\usepackage{amsthm} 
\usepackage{epsfig} 
\usepackage{color}
\usepackage{amscd}
\usepackage[all]{xy}

\newtheorem{lemma}{Lemma}[section]
\newtheorem{teo}[lemma]{Theorem}
\newtheorem{prop}[lemma]{Proposition}
\newtheorem{cor}[lemma]{Corollary}

\theoremstyle{definition}
\newtheorem{defn}[lemma]{Definition}

\newtheorem{quest}[lemma]{Question}

\newtheorem{rem}[lemma]{Remark}

\newcommand{\matN}{\ensuremath {\mathbb{N}}}
\newcommand{\matR} {\ensuremath {\mathbb{R}}}

\newcommand{\matZ} {\ensuremath {\mathbb{Z}}}

\newcommand{\matP} {\ensuremath {\mathbb{P}}}
\newcommand{\matH} {\ensuremath {\mathbb{H}}}

\newcommand{\matRP} {\ensuremath {\mathbb{RP}}}

\newcommand{\calT} {\ensuremath {\mathcal{T}}}

\newcommand{\calH}{\ensuremath {\mathcal{H}}}

\newcommand{\calS}{\ensuremath {\mathcal{S}}}
\newcommand{\calV}{\ensuremath {\mathcal{V}}}

\newcommand{\interior}[1]{{\rm int}(#1)}

\newcommand{\nota} [1] {\caption{\footnotesize{#1}}}

\newcommand{\Isom}{\cong}

\newcommand{\alt}{{\rm alt}}
\newcommand{\inc}{{\rm inc}}
\newcommand{\Full}{{\rm Full}}

\newcommand{\stra} {\ensuremath {{\rm st}}}
\newcommand{\str} {\ensuremath {{\rm str}}}

\newcommand{\mathno}{\ensuremath{\overline{\matH^n}}}

\newcommand{\Tt}{\mathcal{T}}

\newcommand{\Vol}{{\rm vol}}
\newcommand{\Conv}{{\rm Conv}}
\newcommand{\algvol}{{\rm algvol}}
\newcommand{\vol}{{\rm vol}}
\newcommand{\vare}{\varepsilon}

\newcommand{\commento}[1]{}

\author{Stefano Francaviglia}
\address{Dipartimento di Matematica Universit\`a di Bologna, P.zza di Porta S. Donato 5, 40126 Bologna, Italy}
\email{stefano.francaviglia@unibo.it}
\author{Roberto Frigerio}
\address{Dipartimento di Matematica ``L. Tonelli'', Largo Pontecorvo 5, 56127 Pisa, Italy}
\email{frigerio@dm.unipi.it}
\author{Bruno Martelli}
\address{Dipartimento di Matematica ``L. Tonelli'', Largo Pontecorvo 5, 56127 Pisa, Italy}
\email{martelli@dm.unipi.it}

\title[Stable complexity and simplicial volume of manifolds]{Stable complexity \\ and simplicial volume of manifolds}


\begin{document}

\begin{abstract}
Let the $\Delta$-complexity $\sigma(M)$ of a closed
manifold $M$ be the minimal number of simplices in a triangulation of $M$. 
Such a quantity is clearly submultiplicative with respect to finite coverings, and by taking the infimum on all finite coverings of $M$ normalized by the covering degree we can promote $\sigma$ to a multiplicative invariant, a \emph{characteristic number} already considered by Milnor and Thurston, which we denote by $\sigma_\infty(M)$ and call the \emph{stable $\Delta$-complexity} of $M$. 

We study here the relation between the stable $\Delta$-complexity $\sigma_\infty(M)$ of $M$ and Gromov's simplicial volume $\|M\|$. It is immediate to show that $\|M\|\leqslant \sigma_\infty (M)$ and it is natural to ask whether the two quantities coincide on aspherical manifolds with residually finite fundamental group. We show that this is not always the case: there is a constant $C_n<1$ such that $\|M\|\leqslant C_n\sigma_\infty(M)$ for any hyperbolic manifold $M$ of dimension $n\geqslant 4$. 

The question in dimension $3$ is still open in general. We prove that $\sigma_\infty(M) = \|M\|$ for any aspherical irreducible 3-manifold $M$ whose JSJ decomposition consists of Seifert pieces and/or hyperbolic pieces commensurable with the figure-eight knot complement. The equality holds for all closed hyperbolic 3-manifolds if a particular three-dimensional version of the Ehrenpreis conjecture is true. 
\end{abstract}

\maketitle

\section*{Introduction}
Following Milnor and Thurston \cite{MiThu}, a numerical invariant $\alpha(M)$ associated to any closed $n$-manifold $M$ is a \emph{characteristic number} if for every degree-$d$ covering $M\stackrel{d}{\to} N$ we have $\alpha(M) = d\cdot \alpha(N)$. Two important characteristic numbers are the Euler characteristic $\chi(M)$ and the \emph{simplicial volume} $\|M\|$ introduced by Gromov~\cite{Gro}, which equals (up to a constant factor depending on $n$) the volume of $M$ when $M$ is a hyperbolic manifold.

In \cite{MiThu} Milnor and Thurston introduce various characteristic numbers, including the following one. Let $\sigma(M)$ be the \emph{$\Delta$-complexity} of $M$,
\emph{i.e.}~the minimal number of tetrahedra in a triangulation of $M$. We employ here the word ``triangulation'' in a loose sense, as is customary in geometric topology: a triangulation is the realization of $M$ as the glueing of finitely many simplices via some simplicial pairing of their facets. The $\Delta$-complexity is clearly not a characteristic number:
for every degree-$d$ covering $M\stackrel d\to N$ 
we have 
$$\sigma(M)\leqslant d\cdot \sigma(N),$$
but such inequality is very often strict, \emph{i.e.}~we typically get $\sigma(M)<d\cdot \sigma(N)$. We can however easily promote $\sigma$ to a characteristic number as follows. 
We define the \emph{stable $\Delta$-complexity} $\sigma_\infty(M)$ of $M$ by setting
$$\sigma_\infty (M) = \inf_{\widetilde M \stackrel d\to M} \left\{\frac{\sigma(\widetilde M)}d\right\}$$
where the infimum is taken over all finite coverings $\widetilde M \stackrel d\to M$ of any finite degree $d$. Stable $\Delta$-complexity is easily seen to be a characteristic number, that is we have $$\sigma_\infty(M) = d\cdot \sigma_\infty(N)$$ for every finite covering $M\stackrel d\to N$. The characteristic number $\sigma_\infty$ was first defined by Milnor and Thurston in \cite{MiThu}. The following easy inequalities are established 
in Subsection~\ref{easy:sub}.
\begin{prop}\label{easy:prop}
Let $M$ be a closed manifold. We have
$$\|M\|\leqslant \sigma_\infty(M) \leqslant \sigma(M).$$
\end{prop}
The main question we address here is the following:
\begin{quest} \label{main:quest}
For which closed manifolds $M$ do we have $\|M\|= \sigma_\infty (M)$?
\end{quest}

Among the motivations for studying such a problem, we mention the following (open) question of Gromov:

\begin{quest}[\cite{Gromov2}, page 232]\label{gromov:conj} 
Let $M$ be an aspherical closed manifold. Does $\|M\|=0$ imply $\chi (M)=0$? 
\end{quest}

It turns out that if $\|M\| = \sigma_\infty(M)$ on aspherical manifolds then one could easily answer Gromov's question, thanks to the following simple fact, proved below in Subsection \ref{easy:sub}.

\begin{prop} \label{chi:prop}
Let $M$ be a closed manifold. If $\sigma_\infty (M)=0$ then $\chi(M)=0$.
\end{prop}

It is tempting to guess that $\|M\| = \sigma_\infty (M)$ 
at least when $M$ is hyperbolic, because $\pi_1(M)$ is residually finite and hence $M$ has plenty of finite coverings of arbitrarily large injectivity radius. However, we show here that this guess is wrong.

\begin{teo} \label{4:teo}
In every dimension $n\geqslant 4$ there is a constant $C_n<1$ such that $\|M\|\leqslant C_n\sigma_\infty(M)$ for every closed hyperbolic $n$-manifold $M$.
\end{teo}
A similar result holds if we replace stable complexity with stable integral
simplicial volume (see Section~\ref{prel} and Theorem~\ref{integral:pre} below). We mention for completeness the following converse inequality, proved below in Subsection \ref{converse:subsection}.

\begin{prop} \label{converse:prop} In every dimension $n\geqslant 2$ there is a constant $D_n>1$ such that $\sigma_\infty (M) \leqslant D_n\|M\|$ for every closed hyperbolic $n$-manifold $M$.
\end{prop}

Theorem \ref{4:teo} does not hold in dimension two and three. In dimension 2, it is easy to prove that $\sigma_\infty(S) = \|S\|$ for any closed hyperbolic surface $S$ (see Proposition~\ref{surface}). In dimension 3 we have the following.
\begin{teo} \label{3:teo}
There is a sequence $M_i$ of closed hyperbolic 3-manifolds such that 
$$\frac {\sigma_\infty (M_i)}{\|M_i\|} \to 1.$$
\end{teo}

The main difference between dimensions two, three, and higher depends on the fact that the regular ideal hyperbolic $n$-simplex $\Delta^n$ 
can tile $\matH^n$ only in dimensions 2 and 3. The key observation that we use to prove Theorem \ref{4:teo} is that the dihedral angle of $\Delta^n$ does not divide $2\pi$ when $n\geqslant 4$.

In dimension three Question \ref{main:quest} for hyperbolic
3-manifolds remains (as far as we know) open. Our ignorance on this point can be expressed as follows: we do not know any closed hyperbolic 3-manifold $M$ for which $\sigma_\infty(M) = \|M\|$, and we do not know any closed hyperbolic 3-manifold $M$ for which $\sigma_\infty(M)\neq \|M\|$. We refer the reader to Section~\ref{futuro:sec} for a brief discussion about some possible
approaches to, and reformulations of this problem.

We know however various non-hyperbolic 3-manifolds for which Question \ref{main:quest} has a positive answer. 

\begin{teo} \label{JSJ:teo}
Let $M$ be an irreducible manifold with infinite fundamental group, which decomposes along its JSJ decomposition into pieces, each homeomorphic to a Seifert manifold or a hyperbolic manifold commensurable with the figure-8 knot complement. Then
$$\sigma_\infty(M) = \|M\|.$$
\end{teo}

In particular if $M$ is a graph manifold with infinite fundamental group we have $\sigma_\infty (M) = \|M\| = 0$. If $M=DN$ is the double of the complement $N$ of the figure-eight knot we have $\sigma_\infty (M) = \|M\| = 2\|N\| = 4$
(the simplicial volume of bounded manifolds is defined in Section~\ref{prel}).

It is absolutely necessary to restrict ourselves to manifolds with infinite fundamental group, since a manifold $M$ with finite fundamental group has only finitely many coverings and hence $\sigma_\infty(M)>0$, whereas $\|M\|=0$ (since
the simplicial volume is a characteristic number
and vanishes on simply connected manifolds~\cite{Gro, Ivanov}).

To prove Theorem \ref{JSJ:teo} we slightly modify the definition of $\sigma_\infty$ by using \emph{spines} instead of \emph{triangulations} in the spirit of Matveev complexity \cite{Mat}. The resulting invariant, which we denote by $c_\infty (M)$, is another characteristic number defined for any compact manifold $M$ of any dimension, possibly with boundary. When $M$ is an irreducible 3-manifold with infinite fundamental group we get $c_\infty(M) = \sigma_\infty(M)$. On more general 3-manifolds we have $\|M\| \leqslant c_\infty (M) \leqslant \sigma_\infty (M)$ and $c_\infty$ has a better behaviour than $\sigma_\infty$. For instance, we get $c_\infty(M)=0$ on any 3-manifold $M$ with finite fundamental group (in contrast with $\sigma_\infty$) and we can prove the following.

\begin{teo} \label{additive:teo}
The invariant $c_\infty$ is additive on connected sums and on the pieces of the JSJ decomposition.
\end{teo}

Note that Gromov norm is also additive on connected sums and on the pieces of the JSJ decomposition~\cite{Soma}. To deduce Theorem \ref{JSJ:teo} from Theorem \ref{additive:teo} it suffices to check that $c_\infty(M) = \|M\|$ when $M$ is an $S^1$-bundle over a surface or the complement of the figure-eight knot, two special cases that are easy to deal with.

\subsection{Structure of the paper} 
We introduce in Section \ref{prel} the simplicial volume, stable integral volume, and stable $\Delta$-complexity, and prove some basic properties. Section \ref{higher} is devoted to dimension $n\geqslant 4$ and hence to the proof of Theorem \ref{4:teo}. In Section \ref{complexity:section} we introduce the stable complexity $c_\infty$ and 
in Section \ref{Three:section} we turn to 3-manifolds, thus proving Theorems \ref{3:teo}, \ref{JSJ:teo}, and \ref{additive:teo}. Section \ref{futuro:sec} contains some concluding remarks and open questions.

\subsection{Acknowledgements} We thank Clara L\"oh and Juan Souto for useful conversations. 
\section{Preliminaries}\label{prel}
We introduce in this section three characteristic numbers: the well-known \emph{simplicial volume} introduced by Gromov \cite{Gro}, a less-known variation which uses integral homology instead of real homology which we call \emph{stable integral volume}, and the \emph{stable $\Delta$-complexity}, first introduced by Milnor and Thurston in \cite{MiThu} and studied in this paper. A further characteristic number called \emph{stable complexity} uses spines instead of triangulations and is introduced in Section \ref{complexity:section}.

\subsection{Simplicial volume}
Let $M$ be a compact connected oriented $n$-manifold (possibly with boundary), and let $[M,\partial M]^\matZ$ be the integral fundamental class of $M$, \emph{i.e.}~the generator of $H_n(M,\partial M;\matZ)\cong\matZ$ corresponding to the orientation of $M$. 
The inclusion $\matZ\hookrightarrow \matR$ induces
a map $H_n(M,\partial M;\matZ)\to H_n(M,\partial M;\matR)$ which sends
$[M,\partial M]^\matZ$ into the real fundamental class $[M,\partial M]\in H_n(M,\partial M;\matR)$
of $M$.
Following Gromov~\cite{Gro}, we define the \emph{simplicial volume} $\|M\|$
and the \emph{integral simplicial volume} $\|M\|^\matZ$ of $M$ as follows:
\begin{align*}
\|M\| &= \inf\left\{
\sum_{i=1}^k |\lambda_i|\, , \ \left[ \sum_{i=1}^k \lambda_i\sigma_i\right]=[M,\partial M]\,\in\, H_n(M,\partial M;\matR)\,\right\}\ \in\ \matR,\\
\|M\|^\matZ &= \inf\left\{
\sum_{i=1}^k |\lambda_i|\, , \ \left[ \sum_{i=1}^k \lambda_i\sigma_i\right]=[M,\partial M]^\matZ\,\in\, H_n(M,\partial M;\matZ)\,\right\}\ \in\ \matZ .
\end{align*}
The (integral) simplicial volume does not depend on the orientation of $M$ and the 
(integral) simplicial volume of a nonorientable manifold is defined as half the volume of its orientable double covering (hence the integral version may be a half-integer). Moreover, the (integral) simplicial volume of a disconnected manifold is the sum of the simplicial volumes of its components.

As mentioned above, the simplicial volume is a characteristic number, 
\emph{i.e.}~it is multiplicative under finite coverings~\cite{Gro}. On the contrary,
the integral simplicial volume is only submultiplicative: every 
characteristic number vanishes on manifolds that admit finite non-trivial self-coverings,
\emph{e.g.}~on $S^1$, 
while $\| M\|^\matZ\geq 1$ for every closed orientable manifold.
We may therefore
define the \emph{stable} integral simplicial volume $\| M\|^\matZ_\infty$ as follows:
$$\| M\|_\infty^\matZ = \inf_{\widetilde M \stackrel d\to M} \left\{\frac{\| \widetilde M\|^\matZ}d\right\}\ .$$

As observed in Proposition~\ref{last:prop},
the stable integral simplicial volume bounds from above (up to a constant
depending only on the dimension)
the Euler characteristic, so it can be exploited to study Gromov's Question~\ref{gromov:conj}. However, in Section~\ref{futuro:sec}
we will prove the following analog of Theorem~\ref{4:teo}:
\begin{teo}\label{integral:pre}
For every $n\geqslant 4$ there exists a constant $C_n<1$ such that the following holds.
Let $M$ be a closed orientable hyperbolic manifold of dimension $n\geqslant 4$.
Then 
$$
\|M\| \leqslant C_n \|M\|^\matZ_\infty.
$$
\end{teo}
It is folklore that the simplicial volume of a manifold is equal to the
seminorm of its \emph{rational} fundamental class (see~\cite{BFP} for a
complete proof). As a consequence, 
integral cycles may be used to approximate
the simplicial volume via the following equality, which holds for every compact
orientable $n$-manifold $M$:
$$
\| M\|=\inf\left\{
\frac{\sum_{i=1}^k |\lambda_i|}{|h|}\, , \ \left[ \sum_{i=1}^k \lambda_i\sigma_i\right]=h\cdot [M,\partial M]^\matZ\,\in\, H_n(M,\partial M;\matZ),\, h\in\matZ\setminus\{0\}\right\}.
$$
Note however that this equality does not seem to be useful in order to attack Gromov's Question~\ref{gromov:conj}. 

\subsection{Stable $\Delta$-complexity}\label{easy:sub} We work in the PL category, so every manifold in this paper will be tacitly assumed to have a piecewise-linear structure.
As mentioned in the introduction, a \emph{(loose) triangulation} of a closed $n$-dimensional manifold $M$ is the realization of $M$ as the glueing of finitely many $n$-simplices via some simplicial pairing of their facets. The \emph{$\Delta$-complexity} $\sigma(M)$ of $M$ is the minimal number of simplices needed to triangulate $M$. The \emph{stable $\Delta$-complexity} of $M$ is then
$$\sigma_\infty (M) = \inf_{\widetilde M \stackrel d\to M} \left\{\frac{\sigma(\widetilde M)}d\right\}.$$
We can easily establish the inequalities
\begin{equation}\label{easy:ineq}
 \| M\|\leqslant \sigma_\infty (M)\leqslant \sigma (M)
\end{equation}
stated in Proposition~\ref{easy:prop}. The assertion $\sigma_\infty (M)\leqslant \sigma (M)$
follows from the definitions. In order to prove the other inequality we may suppose
that $M$ is oriented. Let $\Tt$ be a triangulation of $M$ with $m=\sigma (M)$ simplices, 
and let  $s_1,\ldots,s_m$ 
be suitably chosen orientation-preserving parameterizations of
the simplices of $\Tt$. We would like to say that $s_1+\ldots + s_m$ represents the fundamental class in $H_n(M;\matZ)$, however this singular chain is not necessarily a cycle. We can fix this problem easily by averaging each $s_i$ on all its permutations. 
That is, we define for any simplex $s$ the chain
$$
\alt(s)=\frac{1}{(n+1)!}\sum_{\tau\in \mathfrak{S}_{n+1}} (-1)^{{\rm sgn}(\tau)}s\circ \overline{\tau},
$$
where $\overline{\tau}$ is the unique affine diffeomorphism of the
standard $n$-simplex $\Delta_n$ corresponding to the permutation $\tau$ of the vertices of $\Delta_n$. Now it is immediate to verify that the chain
$z=\alt(s_1)+\ldots+\alt(s_m)$ is a cycle which represents the fundamental class of $M$.
Moreover, the sum of the absolute values of the coefficients of $z$ is at most $m$,
and this implies the inequality $\| M\|\leqslant \sigma (M)$. The fact that $\| M\|\leqslant\sigma_\infty (M)$ now follows from
the fact that the simplicial volume is multiplicative under finite coverings.

It is also easy to prove a stronger version of Proposition \ref{chi:prop}. 
\begin{prop} \label{chi2:prop}
Let $M$ be a closed $n$-dimensional manifold. We have
$$\big|\chi(M)\big|\leqslant 2^{n+1}\sigma_\infty (M).$$
\end{prop}
\begin{proof}
A triangulation $\Tt$ of $M$ endows $M$ with a cellular structure with at most $2^{n+1}\cdot t$ cells, where $t$ is the number of the simplices of $\Tt$. 
Since the Euler characteristic $\chi(M)$ can be computed
as the alternating sum of the number of simplices in a triangulation of $M$, 
this readily implies that $|\chi(M)|\leqslant 2^{n+1}\sigma (M)$ for every $n$-manifold $M$. Since  $\chi$ is a characteristic number, also the stronger inequality $|\chi(M)|\leqslant 2^{n+1}\sigma_\infty (M)$ holds.  
\end{proof}

\subsection{Surfaces}
In the two-dimensional case 
the answer to Question~\ref{main:quest} is well-known. In fact, considering triangulations
of finite coverings is the standard way to compute the (upper bound for) the simplicial
volume of surfaces of negative Euler characteristic: 
\begin{prop}\label{surface}
Let $S$ be a closed compact surface. If $S=S^2$ (resp.~$S=\matR\matP^2$)
then $\sigma_\infty (S)=2$ (resp.~$\sigma_\infty (S)=1$) and $\| S\|=0$. Otherwise we have
$$
\sigma_\infty (S)=\| S\| = 2 | \chi(S) |.
$$
\end{prop}
\begin{proof}
Of course we have $\sigma (S^2)=2$, so $\sigma_\infty(S^2)=2$ and
$\sigma_\infty (\matR\matP^2)=1$. Moreover, since $S^2$ admits a self-map
of degree bigger than one we have $\|S^2\|=0$, whence $\|\matR\matP^2\|=0$ because
the simplicial volume is a characteristic number.

Let us now suppose that $\chi(S)\leqslant 0$.
Of course, it is sufficient to consider the case when $S$ is orientable. Then,
the equality  
$
\sigma_\infty (S)=\| S\| = 2 | \chi(S) |
$
is well-known (see \emph{e.g.}~\cite{BePe}).
\end{proof}

\subsection{The simplicial volume of hyperbolic manifolds}
The simplicial volume of a manifold is deeply related to several geometric
properties of the Riemannian structures that the manifold can support. Concerning hyperbolic manifolds, the following result due to Gromov and Thurston shows that
the simplicial volume is proportional to the Riemannian volume
(for a complete proof we refer the reader  to~\cite[Theorem C.4.2]{BePe}, \cite[Theorem 11.6.3]{Ratcliffe} 
or~\cite{Bucher} for the closed
case, and \cite{Francaviglia1}, \cite{FriPag} or \cite{FM} for the cusped case).
Let $M$ be a complete finite-volume hyperbolic $n$-manifold. If $M$ is non-compact,
then it admits a natural compactification $\overline{M}$ such that
$\overline{M}$ is a manifold with boundary and $\partial \overline{M}$ is a finite
collection of closed $(n-1)$-manifolds each of which supports a flat Riemannian metric.
We denote by $v_n$ the volume of the ideal regular hyperbolic simplex in $\matH^n$. 
The following result is due to 
Thurston and Gromov~\cite{Thurston, Gro}
(detailed proofs can be found in~\cite{BePe} for the closed case, and in~\cite{Francaviglia1,FriPag,FM}
for the cusped case):

\begin{teo}[Gromov, Thurston]\label{prop:teo}
 Let $M$ be a complete finite-volume hyperbolic manifold with compactification
$\overline{M}$ (so $\overline{M}=M$ if $M$ is closed).
Then
$$
\left\| \overline{M} \right\|=\frac{\Vol (M)}{v_n}.
$$
\end{teo}

\subsection{Converse inequality} \label{converse:subsection}
We prove here Proposition \ref{converse:prop}. The proof 
was communicated to us by Juan Souto, and closely follows ideas of
Thurston~\cite[Theorem 5.11.2]{Thurston} and 
Gromov~\cite[Section 2.1]{Gro}.

\begin{prop} In every dimension $n\geqslant 2$ there is a constant $D_n>1$ such that $\sigma_\infty (M) \leqslant D_n\|M\|$ for every closed hyperbolic $n$-manifold $M$.
\end{prop}
\begin{proof}
Let $R>0$ be any fixed positive real number. Since $\pi_1(M)$ is residually finite we may replace $M$ with a finite cover (which we still call $M$) with injectivity radius bigger than $3R$. 

Let $S\subset M$ be a maximal set of points that are pairwise at distance $\geqslant R$. Consider the Dirichlet tesselation of $M$ into polyhedra determined by $S$, where every point $x_0\in S$ gives rise to the polyhedron 
$$P_{x_0} = \big\{y \in M\ \big|\ d(x_0,y) \leqslant d(x, y) \ \forall x\in S \big\}.$$
This is indeed isometric to a convex polyhedron because the injectivity radius of $M$ is sufficiently big.
We have $B\left(x_0, \frac R2\right) \subset P_{x_0} \subset B\left(x_0, R\right) $. The number of polyhedra is therefore bounded above by
$$\frac {\Vol (M)}{\Vol \left(B\left(x_0, \frac R2\right)\right)}.$$
A facet $F$ of $P_{x_0}$ corresponds to some point $x_F\in S$ such that $d(x,x_0) = d(x,x_F)$ for all $x\in F$; since $d(x_0,x_F) < 2R$ the number of facets of $P_{x_0}$ is smaller or equal than the number of points in $S \cap B(x_0, 2R)$, which is in turn smaller or equal than the ratio between $\Vol \left(B\left(x_0, 3R\right)\right)$ and $\Vol \left(B\left(x_0, \frac R2\right)\right)$.

Therefore the number of facets of each $P_x$ is uniformly bounded and hence the possible combinatorial types for $P_x$ vary on a finite set which depends only on the dimension $n$ and $R$. Choose for each possible combinatorial type a triangulation which induces on every facet a triangulation that is symmetric with respect to every combinatorial isomorphism of the facet: these symmetric triangulations necessarily match to give a triangulation of $M$. Let $T$ be the maximal number of simplices of the triangulated combinatorial types. Our original manifold $M$ triangulates with at most 
$$\frac {T}{\Vol \left(B\left(x_0, \frac R2\right)\right)} \cdot \Vol (M)  = \frac {Tv_n}{\Vol \left(B\left(x_0, \frac R2\right)\right)} \cdot \| M \|  $$
simplices.
\end{proof}

\section{Higher dimensions} \label{higher}
This section is devoted to the proof of 
following theorem. 
\begin{teo}\label{bah}
For every $n\geqslant 4$ there exists a constant $C_n<1$ such that the following holds.
Let $M$ be an $n$-dimensional closed orientable hyperbolic manifold.
Then
$$
\Vol(M) \leqslant C_n v_n \sigma(M).
$$
\end{teo}

Putting toghether this result with Theorem~\ref{prop:teo} we get the following.

\begin{cor}\label{main:cor} 
We have $\|M\| \leqslant C_n \sigma(M)$ for every closed orientable hyperbolic $n$-manifold $M$ of dimension $n\geqslant 4$.
\end{cor}
Since the simplicial volume is a characteristic number, Corollary~\ref{main:cor}
implies in turn Theorem~\ref{4:teo}.

\begin{cor}
We have $\|M\| \leqslant C_n \sigma_\infty(M)$ for every closed hyperbolic $n$-manifold $M$ of dimension $n\geqslant 4$.
\end{cor}

\subsection{Straight simplices}
We recall that every pair of points $\mathno$ is connected by a unique geodesic segment (which 
has infinite length if any of its endpoints lies in $\partial\mathno$). A subset in $\mathno$ is \emph{convex} if whenever it contains a pair of points it also contains the geodesic segment connecting them. The \emph{convex hull} of a set $A$ is defined as usual as the intersection of all convex sets containing $A$.

A \emph{(geodesic) $k$-simplex} $\Delta$ in $\mathno$ is the convex hull of $k+1$ points in $\mathno$, called \emph{vertices}. A $k$-simplex is:
\begin{itemize}
\item \emph{ideal} if all its vertices lie in $\partial\matH^n$,
\item \emph{regular} if every permutation of its vertices is induced by an isometry of $\matH^n$,
\item \emph{degenerate} if it is contained in a $(k-1)$-dimensional subspace of $\matH^n$.
\end{itemize}

Let $v_n$ be the volume of the regular ideal simplex in $\mathno$. 

\begin{teo}[\cite{HM, Pe}] \label{maximal:teo}
Let $\Delta$ be a geodesic $n$-simplex in $\mathno$. Then $\Vol(\Delta)\leqslant v_n$, and $\Vol(\Delta) = v_n$ if and only if $v_n$ is ideal and regular.
\end{teo}

A \emph{singular $k$-simplex} in $\matH^n$ is of course a continuous map $\sigma\colon \Delta_k \to \matH^n$
from the standard $k$-simplex $\Delta_k\subset \matR^{k+1}$ to hyperbolic space. The corresponding \emph{straight} simplex $\sigma^\stra\colon \Delta_k \to \matH^n$ is defined as follows: set
 $\sigma^\stra(v) = \sigma(v)$ on every vertex $v$ of $\Delta_k$, and
extend using barycentric coordinates (which exist in $\matH^n$, using
the hyperboloid model). The image of $\sigma^\stra$ is the convex hull
of the images of the vertices of $\Delta_k$ via $\sigma$, hence it is a geodesic simplex.

Using again barycentric coordinates, for every singular $k$-simplex in $\matH^n$ we can define a
homotopy $H(\sigma)\colon \Delta_k\times [0,1]\to \matH^n$ 
between $\sigma$ and $\sigma^\stra$ by setting $H(\sigma)(p,t)= 
t\sigma (p)+(1-t)\sigma^\stra(p)$. The following lemma readily descends from the definitions
and from
the fact that barycentric coordinates commute with the isometries of $\matH^n$.

\begin{lemma}\label{invariance}
 Let $\sigma\colon\Delta_k\to\matH^n$ be a singular simplex, and let $g$ be an isometry
of $\matH^n$. Then:
\begin{enumerate}
 \item $(g\circ\sigma)^\stra=g\circ\sigma^\stra$ and $H(g\circ \sigma)=g\circ H(\sigma)$;
\item
if $h<k$ and  $i\colon \Delta_{h}\to \Delta_k$ is an affine inclusion of $\Delta_h$ onto an
$h$-dimensional face
of $\Delta_k$, then $(\sigma\circ i)^\stra=\sigma^\stra\circ i$ and
$H(\sigma\circ i)=H(\sigma)\circ (i\times {\rm Id})$.
\end{enumerate}
\end{lemma}

Henceforth, $M$ will always be an oriented hyperbolic closed $n$-dimensional manifold. 
Let $\sigma\colon \Delta_k\to M$ be a singular simplex in $M$. The straightening $\sigma^\stra\colon\Delta_k \to M$ 
is defined by lifting the map $\sigma$ to the universal covering $\matH^n$, straightening it, and then projecting it back to $M$. 
By  Lemma~\ref{invariance}
this operation does not depend on the chosen lift.

The {\em algebraic volume} of a singular $n$-simplex $\sigma\colon\Delta_n\to M$ is   
$$\algvol(\sigma)=\int_{\sigma^\stra}d\vol = \int_{\Delta_n} (\sigma^\stra)^*d\vol$$
where $d\vol$ is the volume form on $M$.
The absolute value $|\algvol(\sigma)|$ equals the volume of the image of any lift 
$\widetilde{\sigma}^\stra$ of $\sigma^\stra$ to $\matH^n$. In particular,
$\algvol(\sigma)$ vanishes if and only if $\widetilde\sigma^\stra$ is degenerate. When $\algvol(\sigma)\neq 0$ the straightened singular simplex $\sigma^\stra$ is an immersion and the sign of $\algvol(\sigma)$ depends on whether
$\sigma^\stra$ is orientation-preserving or not.

As we said above, every geodesic $n$-simplex in $\matH^n$ has volume smaller
than the volume $v_n$ of the regular ideal simplex. In
particular we always have 
$$-v_n\leqslant\algvol(\sigma)\leqslant v_n.$$ 
\begin{defn}
A singular $n$-simplex $\sigma$ in $M$ is {\em positive}
if $\algvol(\sigma)>0$, {\em negative} if $\algvol(\sigma)<0$,
and {\em flat} if $\algvol(\sigma)=0$. 
For $\varepsilon>0$, the singular simplex $\sigma$ is 
\emph{$\varepsilon$-big} if $$ \algvol (\sigma) \geqslant
(1-\varepsilon)v_n,$$ 
and \emph{$\varepsilon$-small} otherwise. 
\end{defn}

\subsection{The straightening as a map}\label{stmap}
Let now $\Tt$ be a (loose) triangulation of an oriented hyperbolic closed manifold $M$,
that is the realization of $M$ as the union of $m$ copies of the
standard simplex $\Delta_n$ quotiented by an orientation-reversing
simplicial pairing of their $(n-1)$-dimensional faces. Every simplex
in $\Tt$ is a copy of $\Delta_n$ and hence is the image of an
orientation-preserving singular simplex $\sigma_i\colon \Delta_n\to
M$. Henceforth, if $\sigma$ is a singular simplex in $M$,
we denote by $|\sigma|\subseteq M$ the image of $\sigma$ in $M$.

We define a map 
$$\str_\Tt \colon M\to M$$ 
which corresponds to the simultaneous straightening of all the simplices of $\Tt$. If $p\in M$
lies in  $|\sigma_i |$, we choose a point $q\in \Delta_n$ such
that $\sigma_i(q)=p$ and set $\str_\Tt(p)=\sigma_i^\stra(q)$. Of course, if the point $p$ 
belongs to the $(n-1)$-skeleton of $\Tt$, then 
both the choice of $\sigma_i$ and/or the choice of the point $q\in\Delta_n$ 
are somewhat arbitrary. However, Lemma~\ref{invariance} ensures that $\str_\Tt$ is well-defined, continuous and homotopic to the identity
of $M$. 
In what follows, when a triangulation $\Tt$ is fixed and no ambiguities can arise, we will denote
the map $\str_\Tt$ simply by $\str$.

It is important now to note that
the straightened simplices of $\Tt$ do not necessarily form a triangulation of
$M$ in any reasonable sense: straightened simplices may degenerate
and overlap (and they often do, see 
also Remark \ref{rem1}). However, one important property
is preserved by the straightening: the positive simplices still cover the manifold $M$.

\begin{lemma}\label{algvol}
Let $\sigma_1,\ldots,\sigma_t$ be the simplices of a triangulation $\Tt$ of $M$. Then
$$
M=\bigcup_{{\rm positive\ }\sigma_i} |\sigma_i^\stra|= \str_\Tt\left(\bigcup_{{\rm positive\ } \sigma_i} |\sigma_i|\right) ,
$$
so
$$
\Vol(M)\leqslant \sum_{{\rm positive\ }\ \sigma_i} \Vol\big(|\sigma_i^\stra|\big).
$$
\end{lemma}
\begin{proof}
Let $M_0\subseteq M$ be the image along $\str_\Tt$ of the $(n-1)$-skeleton 
of $\calT$ and of the flat simplices of $\Tt$. The complement $M\setminus M_0$ is open and dense and consists of topologically regular values for the map $\str_\Tt$, that is the pre-image of every point in $M\setminus M_0$ consists of finitely many points where $\str_\Tt$ is a local homeomorphism and has hence local degree $\pm 1$. Since $\str_\Tt$ has globally degree one, every topologically regular value lies in the image of at least one positive simplex.
The conclusion follows since the image via $\str_\Tt$ of the positive simplices of $\Tt$
is compact, whence closed.
\end{proof}

\subsection{Strategy of the proof of Theorem \ref{bah}.}
We outline here the proof of Theorem \ref{bah}. Let $\Tt$ be a triangulation of a closed hyperbolic manifold $M$ of dimension $n\geqslant 4$. We need to prove that $\vol (M) \leqslant C_nv_n t$ where $t$ is the number of simplices in $\Tt$ and $C_n<1$ is a constant depending only on the dimension $n$.

Suppose for simplicity that every simplex of $\Tt$ is positive. In that lucky case the map $\str_\Tt$ is a homeomorphism and the straightened triangulation is a genuine triangulation (which we still denote by $\Tt$) made of straight positive simplices. The key observation now is that in dimension $n\geqslant 4$ the ratio between $2\pi$ and the dihedral angle of an ideal regular geodesic simplex is not an integer. Therefore we may choose $\varepsilon_n>0$ independently of $\Tt$ in such a way that every $(n-2)$-dimensional face $E$ of $\Tt$ enjoys the following properties: 
\begin{enumerate}
\item
the face $E$ is contained in at least one $\varepsilon_n$-small simplex of $\Tt$, and 
\item the number of $\varepsilon_n$-big simplices of $\Tt$   that contain $E$ is uniformly bounded from above
by a universal constant. 
\end{enumerate}
These facts easily imply that the ratio between the number of $\varepsilon_n$-big simplices
of $\Tt$ and the total number $t$ of simplices of $\Tt$ is smaller than some constant $K_n< 1$ independent of $\Tt$. Therefore the volume of $M$ is smaller than 
$$t\big(v_nK_n + (1-\varepsilon_n)v_n\big(1-K_n)) = tv_n\left(K_n + (1-\varepsilon_n)(1-K_n)\right) =  tv_nC_n$$ 
with $C_n = 1-\varepsilon_n(1-K_n)<1$.

We now need to refine this strategy to deal with negative and flat simplices. As we said above, the straightening of $\Tt$ may create degenerations and overlappings of simplices. Degenerations and overlappings are volume-consuming, so
it is reasonable to expect that the inequality $\vol(M)\leqslant C_nt$ holds \emph{a fortiori} in presence of negative and flat simplices: the generalization of the above argument however is not immediate. 

Note for instance that both points (1) and (2) stated above do not hold for a general triangulation $\Tt$: a codimension $(n-2)$ face $E$ may be incident to arbitrarily many arbitrarily big positive simplices, that wind many times around $E$ (the local degree of the straightening map around $E$ can be arbitrarily big! See Figure \ref{degree:fig}, which
is inspired by~\cite[Example 2.6.4]{Francaviglia2}, \cite[Example 4.1]{Francaviglia3}). 
Of course by winding many times around $E$ the simplices overlap a lot and hence a lot of volume is wasted: we will need to estimate that 
loss of volume to prove our theorem.

\begin{figure}
 \begin{center}
  \includegraphics[width = 6.5 cm]{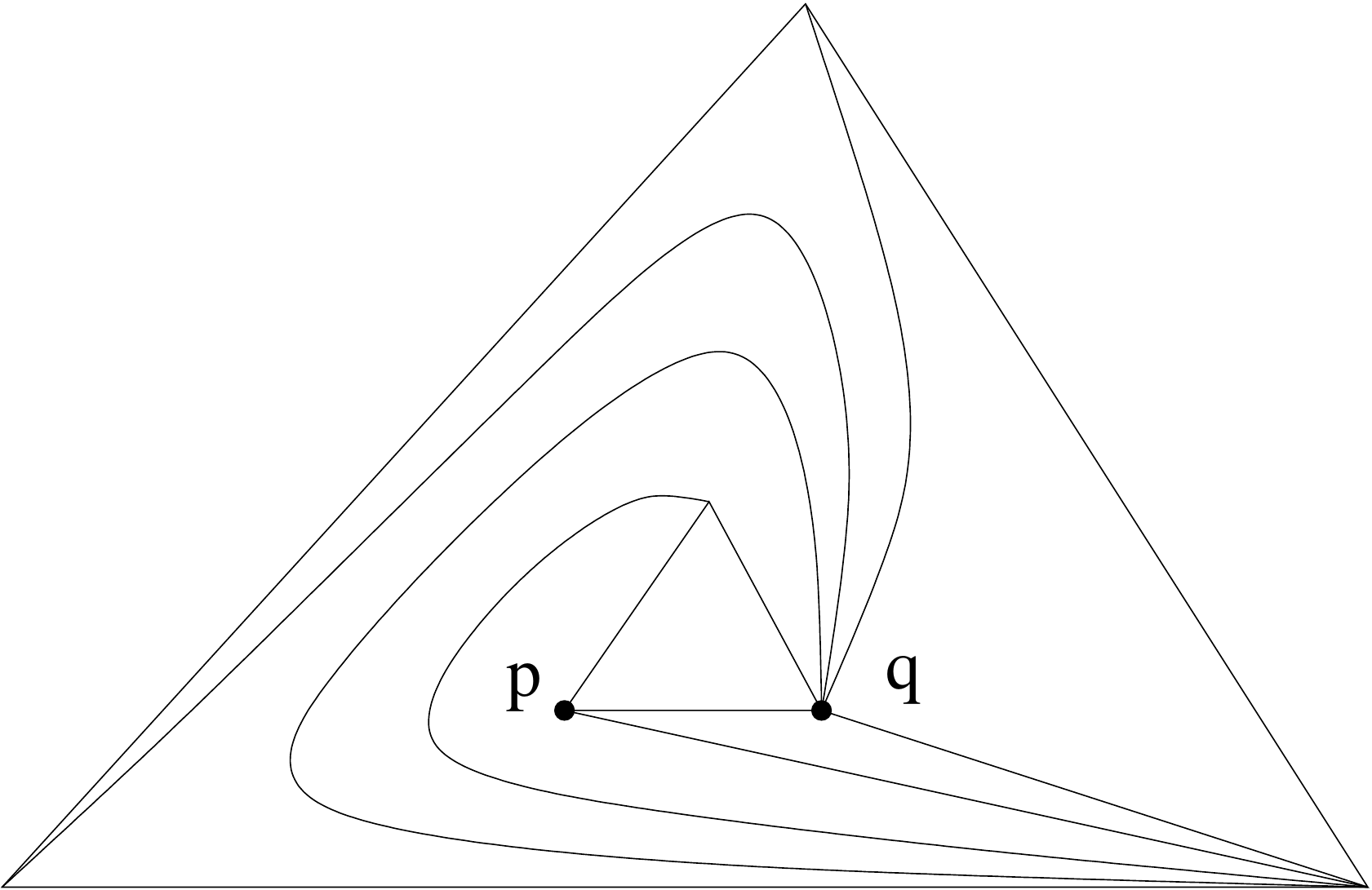}
 \end{center}
 \nota{The local degree of the straightening map associated to the trangulation described here
(which may be thought as a triangulation of a portion of the projective model of
the hyperbolic plane) is equal to $0$ in $p$ and to $2$ in $q$.}
 \label{degree:fig}
\end{figure}

\subsection{The volume of a simplex}\label{simplices:sub}
We will need to estimate (from below) the overlapping regions of big simplices. To do so we first study their geometry.

For every $n\geqslant 3$ and $k\leqslant n$, we denote by 
$\calV_k(\mathno)$ the space of unordered $(k+1)$-tuples of (not necessarily distinct)
points of $\mathno$, \emph{i.e.}~the topological space 
$(\mathno)^{k+1}/\mathfrak{S}_{k+1}$, where $\mathfrak{S}_{k+1}$ is the permutation group on $k+1$ elements.

We also denote by
$\calS_k(\mathno)$ the set of $k$-dimensional geodesic simplices
of $\mathno$, and we endow $\calS_k(\mathno)$
with the topology induced by the Hausdorff topology on closed
subsets of $\overline{\matH^n}$. 
The convex hull defines a surjective map
$$\Conv\colon \calV_k(\mathno) \to \calS_k(\mathno).$$
We will often use the following notation.
\begin{defn}
If $K$ is any subset of $\mathno$, we denote by $H(K)$ the smallest geodesic subspace of $\mathno$ containing $K$
(if $K$ consists of a single point of $\partial\matH^n$, then we set $H(K)=K$).
\end{defn}
Of course an element $K$ in $\calV_k(\mathno)$ or $\calS_k(\mathno)$ is degenerate if
$\dim H(K)<k$. We denote by $\calV_k^*(\mathno)$ and $\calS_k^*(\mathno)$ the set of nondegenerate elements of $\calV_k(\mathno)$ and $\calS_k(\mathno)$. The proof of the following easy result is left to the reader:

\begin{lemma}\label{ovvio}
The map 
$$
\Conv\colon \calV^*_k(\mathno)\to \calS^*_k(\overline{\matH^n})
$$
is a homeomorphism.
\end{lemma}
We are mainly interested in the behaviour of the function
$$
\Vol\colon \calS_k(\overline{\matH^n})\to\matR
$$
which maps every geodesic $k$-simplex into its $k$-dimensional volume. 

Despite its natural definition, the function
$\Vol$ is not continuous on the whole $\calS_k(\overline{\matH^n})$: for example,
let $K$ be any ideal regular simplex and 
$g\in{\rm Isom}(\matH^n)$ be a parabolic isometry that fixes an ideal vertex $p$ of $K$. 
Then $\lim_{i\to\infty} g^i(K) = \{p\}$ and therefore 
$$\lim_{i\to\infty} \Vol(g^i(K))=v_n\neq 0=\Vol \left(\lim_{i\to\infty} g^i(K)\right).$$ 
This shows that in general some care is needed in studying geometric properties of limits of simplices.
However, the following lemma  
ensures that the volume function is continuous on the space of nondegenerate simplices.

\begin{lemma}\label{cont-vol}
The restriction of
$\Vol$ to the set of simplices with at least three different vertices
is continuous. In particular, the restriction
$$
\Vol\colon \calS^*_k(\overline{\matH^n})\to\matR^+
$$
is continuous.
\end{lemma}
\begin{proof}
By Lemma~\ref{ovvio}, the conclusion is an immediate consequence  
of~\cite[Proposition 4.1]{Luo} (see also~\cite[Theorem 11.4.2]{Ratcliffe}).
\end{proof}

\subsection{The incenter and inradius of a simplex}
Lemma~\ref{cont-vol} implies in particular that if a sequence $K_i$ of elements in $\calS_k^*(\mathno)$ converges to an ideal regular $k$-simplex then $\Vol(K_i)\to v_k$ as $i\to\infty$. We are interested in proving the converse result:
the shape of a simplex with  large volume has to be similar to the shape of a regular ideal simplex. However, we have observed above
that a sequence of ideal regular simplices may well converge to a degenerate simplex, so some care is needed here.

Consider a nondegenerate $k$-simplex $K\in \calS_k^*(\overline{\matH^n})$. 
For every point $p\in K\cap\matH^n$ we denote by $r_K(p)$ the radius of the maximal $k$-ball of $H(K)$ centered in $p$ and contained
in $K$. 
Since the volume of any $k$-simplex is smaller than $v_k$ and
the volume of $k$-balls diverges as the radius diverges, there exists a constant $r_k>0$ such that
$r_K(p)\leqslant r_k$ for every $K\in \calS_k^*(\overline{\matH^n})$ and $p\in K$. 

\begin{defn}
Take $K\in \mathcal{S}^*_k(\mathno)$. 
The \emph{inradius} $r(K)$ of $K$ is 
$$
r(K)=\sup_{p\in K\cap \matH^n} r_K(p)\ \in \ (0,r_k]
$$ 
(observe that $r(K)>0$ since $K$ is nondegenerate).
The \emph{incenter} $\inc(K)$ is the unique point $p\in K\cap \matH^n$ such that $r_K(p)=r(K)$.
\end{defn}

\begin{lemma}\label{incentri}
The incenter is well-defined. The sphere centered in $\inc(K)$ of radius $r(K)$ is tangent to all the facets of $K$. The functions
$$
\inc\colon S_k^*({\mathno})\to \matH^n, \qquad
r\colon S_k^*({\mathno})\to\matR
$$
are continuous.
\end{lemma}
\begin{proof}
The map $p\mapsto r_K(p)$ is continuous, and if $q$ is a (possibly ideal) vertex
of $K$ we have $\lim_{p\to q} r_K(p)=0$. Therefore,
the map $r_K\colon K\to (0,r_k]$ is proper, and this ensures
the existence of a point $p\in K$ such that
$r_K(p)=r(K)$. 

Let $S_p$ be the sphere centered in $p$ of radius $r(K)$: we prove that $S_p$ is tangent to every $(k-1)$-face of $K$.
Assume by contradiction that  $F$ is a $(k-1)$-face of $K$ such that
$S_p\cap F=\emptyset$, and denote by $v$ the (possibly ideal) vertex of $K$ opposite to $F$.
Let $\gamma$ be 
the geodesic ray (or line, if $v$ is ideal) exiting from $v$ and containing $p$.
It is readily seen that the distance between $\gamma(t)$ and any $(k-1)$-face of $K$ distinct from $F$ is an increasing function of $t$. If $p=\gamma(t_0)$, then this implies that there
exists $\varepsilon>0$ such that $r_K(\gamma(t_0+\varepsilon))>r_K(p)=r(K)$, a contradiction.

We exploit the hyperboloid model of $\matH^n$ to determine the point $p$ more explicitly, \emph{i.e.}~we 
fix the identification 
$$\mathbb H^n=\big\{w=(w_0,\dots,w_n)\in\mathbb R^{n+1}:
\langle w,w\rangle=-1,\ w_0>0\big\}$$ 
where $\langle\cdot,\cdot\rangle$
denotes the usual Minkowski product.
Let $H$ be the $(k+1)$-dimensional linear subspace of $\matR^{n+1}$ containing $H(K)$.
If $F_0,\ldots,F_k$ are the $(k-1)$-faces of $K$, 
for every $i=0,\ldots,k$ we denote by $q_i$ the dual vector of $F_i$, \emph{i.e.}~
the unique vector $q_i\in H$ such that $\langle q_i,q_i\rangle=1$, 
$\langle q_i,w\rangle=0$ for every $w\in F_i$, and 
 $\langle q_i,w\rangle\leqslant 0$ for every $w\in K$. 
 
 If $w$ is any point of
$K$, then the hyperbolic distance 
between $w$ and the geodesic
$(k-1)$-plane containing $F_i$ satisfies the equality
$$\sinh d(w,H(F_i))=-\langle w,q_i\rangle.$$ 
Let now $H_{ij}\subseteq H$ be the hyperplane of $H$ which is orthogonal
to $q_i-q_j$. Recall that our point $p\in H$ lies at the same distance
from the geodesic planes containing the faces of $K$, so
\begin{equation}\label{inc:eq}
p\in \bigcap_{i\neq j} H_{ij}. 
\end{equation}
Since $K$ is nondegenerate, the vectors $q_0-q_i$, $i=1,\ldots,k$ are linearly independent,
and this readily implies that $\bigcap_{i\neq j}^k H_{ij}$ is a $1$-dimensional
linear subspace of $H$. Such a subspace cannot meet the hyperboloid
$\matH^n$ in more than one point, and this concludes the proof that $p$ is the
unique point of $K$ such that $r_K(p)=r(K)$. Moreover, we have
\begin{equation}\label{inr:eq}
\sinh r(K)=-\langle \inc(K),q_i\rangle\qquad \textrm{for\ every}\ i=0,\ldots,k.
\end{equation}

This description of $\inc(K)$ also implies that $\inc(K)$ and $r(K)$ continuously depend
on $K$. In fact, even when considering simplices with possibly ideal vertices,
it is readily seen that each subspace $H_{i}$, whence each $q_i$ and each $H_{ij}$, 
continuously depends on $K$. Thanks to
equations~\eqref{inc:eq} and~\eqref{inr:eq}, this implies that 
the maps $\inc\colon S_k^*(\mathno)\to \matH^n$
and $r\colon S_k^*({\mathno})\to\matR$
are continuous. 
\end{proof}

We now need the following result proved by Luo.

\begin{lemma}[Proposition 4.2 in \cite{Luo}]\label{liminf}
Let $K_i$ be a sequence of elements in $S_k^*(\mathno)$ such that
$
\lim_{i\to \infty} r(K_i)=0.
$
Then
$
\lim_{i\to \infty} \Vol(K_i)=0
$.
\end{lemma}

We can finally prove that simplices of large volume are close to ideal regular simplices.

\begin{prop}\label{regular-converge}
Let $K_\infty\in S_k^*(\mathno)$ be a fixed ideal regular simplex, and
let $K_i$ be a sequence of elements in $S_k^*(\mathno)$ such that
$$
\lim_{i\to\infty} \Vol (K_i)=v_k.
$$
Then there exists a sequence $g_i$ of isometries of $\matH^n$ such that 
$$
\lim_{i\to\infty} g_i(K_i)=K_{\infty}.
$$ 
\end{prop}
\begin{proof}
We consider 
the disc model for $\matH^n$ and suppose that the origin $O$ is the incenter of $K_\infty$. Let $H$ be the $k$-space containing $K_\infty$. We define the distance $d(K,K')$ of two simplices as the Hausdorff distance with respect to the Euclidean metric of the closed disc. For each $i$ we pick an isometry $g_i$ of $\matH^n$ such that:
\begin{enumerate}
\item $g_i(K_i)$ has its incenter in $O$ and is contained in $H$,
\item 
$g_i$ is chosen among all isometries $g_i$ satisfying (1) in such a way that
$g_i(K_i)$ has the smallest possible distance from $K_\infty$  (such a choice is possible 
since the set of isometries of $\matH^n$ taking $\inc(K_i)$ to $O$ and the geodesic subspace $H(K_i)$ into $H$ is 
homeomorphic to $O(k)$ and hence compact).
\end{enumerate}

Since $\calS_k(\mathno)$ is compact, in order to conclude it is sufficient to show that every converging subsequence of $g_i(K_i)$
converges to $K_\infty$. So, let us take a subsequence 
 which converges to some $k$-simplex $K_\infty'$. Lemma~\ref{liminf} ensures that the sequence of radii $r(K_i)$ is bounded below
 by a positive number, hence the intersection $\cap_ig_i(K_i)$ contains a $k$-ball $B\subset H$ centered in $O$. Therefore $B\subset K_\infty'$ and hence $K_\infty'$ is nondegenerate. We may now apply Lemma~\ref{cont-vol} and get 
$$\Vol (K_\infty')=\lim_{i\to\infty} \Vol(K_i)=v_k.
$$
By Theorem~\ref{maximal:teo}, the simplex $K_\infty'$ is ideal and regular, and assumption (2) easily implies that $K_\infty' = K_\infty$. 
This concludes the proof.
\end{proof}

It is now easy to prove that in big simplices the incenter of a face is uniformly distant from any other non-incident face.

\begin{lemma}\label{palle}
Let $n\geqslant 3$. There exist $\varepsilon_n>0$ and $\delta_n>0$ such
that the following holds for any simplex $\Delta\in \calS_n(\overline{\matH^n})$ 
with $\Vol(\Delta)\geqslant v_n(1-\varepsilon_n)$. Let $E$ be any face of $\Delta$ and $E'$ another face of $\Delta$ which does not contain $E$. Then
$$d(\inc(E), E') > 2\delta_n.$$
\end{lemma}
\begin{proof}
There is only one regular ideal $n$-dimensional simplex $\Delta^{\rm reg}$ up to isometries of $\mathno$. Let $3\delta_n>0$ be the minimal distance between $\inc(E)$ and $E'$ among all pairs of faces $E,E'$ of $\Delta^{\rm reg}$ such that $E\not\subseteq E'$. 

We claim that there is a constant $\varepsilon_n>0$ such that $d(\inc (E),E')>2\delta_n$ for any pair of faces $E\not\subseteq E'$ of any $n$-simplex $\Delta$ of volume bigger than $v_n(1-\varepsilon_n)$. Suppose by contradiction that there is a sequence $\Delta_i$ of $n$-simplices 
with $\lim_{i\to\infty} \Vol(\Delta_i)=v_n$, each $\Delta_i$ containing two faces $E_i \not\subseteq E_i'$ with $d(\inc (E), E') \leqslant 2\delta_n$. By Proposition~\ref{regular-converge}, up to replacing each $\Delta_i$ with an isometric copy we may assume that $\lim_{i\to\infty} \Delta_i=\Delta^{\rm reg}$, $\lim_{i\to\infty} E_i=E_*$, and $\lim_{i\to\infty} E'_i=E'_*$ for some faces $E_* \not\subseteq E_*'$ of $\Delta^{\rm reg}$.
Using Lemma~\ref{incentri}
we get
$$
\lim_{i\to\infty} d(\inc(E_i), E'_i) =
d(\inc(E_*),E_*') \geqslant 3\delta_n
$$ 
hence a contradiction.
\end{proof}

\subsection{Dihedral angles}
Let $\Delta\in \calS_n^*(\overline{\matH^n})$ be a nondegenerate $n$-simplex, and let $E$ be an $(n-2)$-dimensional face of $\Delta$. The \emph{dihedral angle}
$\alpha (\Delta,E)$ 
of $\Delta$ at $E$ is defined as usual in the following way: let $p$ be a point in $E\cap \matH^n$, and let $H\subseteq \matH^n$ be the unique $2$-dimensional geodesic plane which intersects orthogonally $E$ in $p$. We set $\alpha(\Delta,E)$ to be equal to the angle in $p$ of the polygon
$\Delta\cap H$ of $H\cong \matH^2$. It is easily seen that this is well-defined (\emph{i.e.}~independent of $p$).
For every $n\geqslant 3$, 
we denote by $\alpha_n$ the dihedral angle of the ideal regular $n$-dimensional simplex at any of its $(n-2)$-dimensional faces.

It is readily seen by intersecting the simplex with a horosphere
centered at any vertex that $\alpha_n$ equals the dihedral angle of
the regular  \emph{Euclidean} $(n-1)$-dimensional simplex at any of its
$(n-3)$-dimensional faces, so $\alpha_n=\arccos \frac 1{n-1}$. 
In particular, we have $\alpha_2=\arccos\frac{1}{2}=\pi/3$. Moreover,
it is easily checked that $\frac{2\pi}{6}<\arccos \frac{1}{3}<\frac{2\pi}{5}$ and
$\frac{2\pi}{5}<\arccos \frac{1}{n}<\frac{2\pi}{4}$ for every $n\geqslant 4$.
As a consequence, the real number $\frac{2\pi}{\alpha_n}$ is an integer if and only if $n=3$,
and if we denote by
$k_n\in\matN$, $n\geq 4$,  
the unique
integer such that 
$$k_n\alpha_n<2\pi<(k_n+1)\alpha_n,$$ 
then $k_n=5$ if $n=4$ and $k_n=4$ if $n\geqslant 5$.

\begin{lemma}\label{maximal-angle}
Let $n\geqslant 4$. Then, there exist $a_n>0$ 
and $\vare_n>0$, depending only on $n$, such that the 
following condition holds:
if $\Delta\in \calS_n^*(\mathno)$ is an $n$-simplex such that
$\Vol(\Delta)\geqslant (1-\vare_n)v_n$ and
$\alpha$ is the dihedral angle of $\Delta$
at any of its $(n-2)$-faces, then
$$
\frac{2\pi}{k_n+1}(1+a_n) < \alpha < \frac{2\pi}{k_n}(1-a_n).
$$
\end{lemma}
\begin{proof}
It is very easy to show that the dihedral angles of a nondegenerate $n$-simplex
continuously depend on its vertices, so the conclusion follows from 
Proposition~\ref{regular-converge} and the fact that
$\frac{2\pi}{k_n+1} < \alpha_n < \frac{2\pi}{k_n}$. 
\end{proof}

\subsection{Proof Of Theorem~\ref{bah}}\label{proof:sub}
In this subsection we suppose that $M$ is a closed orientable hyperbolic manifold of dimension
$n\geqslant 4$. We will prove that there exists a constant $C_n<1$, only depending
on $n$, such that if $\Tt$ is any triangulation of $M$
with $|\Tt|$ simplices, then $\Vol(M)\leqslant C_n  v_n |\Tt|$.

Let us suppose that $|\Tt|=t$, and let us denote by $\sigma_1,\ldots,\sigma_t$ 
suitably chosen orientation-preserving parameterizations of
the simplices of $\Tt$. Let us fix positive constants $\vare_n,\delta_n$ and $a_n$ that
satisfy the conclusions of Lemma~\ref{palle} and Lemma~\ref{maximal-angle}.

Recall that an $n$-simplex of $\Tt$ is \emph{$\varepsilon$-big} if
$\algvol (\sigma) \geqslant (1-\varepsilon)v_n$. Let $t_b$ and $t_s$ be respectively the number of $\varepsilon_n$-big and $\varepsilon_n$-small simplices in $\Tt$, so that $t = t_b+ t_s$.
We begin with the following easy

\begin{lemma}\label{stima1a}
 Suppose that $t_s\geqslant \frac t{12}$. Then
$$
\Vol (M)\leqslant 
\left(1-\frac{\vare_n}{12}\right) t v_n.
$$
\end{lemma}
\begin{proof}
Our assumption implies that $t_b+(1-\vare_n)t_s=t-\vare_n t_s\leqslant
(1-\frac{\vare_n}{12})t$. Moreover, 
if $\sigma_i$ is $\vare_n$-small, then 
either it is negative or 
$\vol (|\sigma_i^\stra|)\leqslant
(1-\vare_n)v_n$. Therefore 
Lemma~\ref{algvol} implies that
$$
\Vol(M)\leqslant \sum_{\textrm{positive}\ \sigma_i} \vol(|\sigma^\stra_i|)\leqslant v_n(t_b+(1-\vare_n)t_s)
\leqslant \left(1- \frac{\vare_n}{12}\right) tv_n.
$$
\end{proof}
Therefore if $t_s\geqslant \frac t{12}$ we are done: henceforth we assume that $t_s\leqslant \frac t{12}$.

If $E\subseteq M$ is an $(n-2)$-dimensional face of $\Tt$, we 
denote by $v(E)$ the number of $\vare_n$-big simplices (counted with multiplicities) of $\Tt$
which are incident to $E$.
We say that
$E$ is \emph{full} if $v(E)\geqslant k_n+1$, 
we denote by $\Full(\Tt)$ the set of full $(n-2)$-dimensional faces of
$\Tt$, and we set 
$$e_{\rm f}=| \Full(\Tt)|,\qquad N=\sum_{E\in \Full(\Tt)} v(E).$$ 

\begin{rem}\label{rem1}
Observe that if $E$ is full, then $\str(E)$ is a face of
a nondegenerate $n$-simplex, so it is itself nondegenerate. In particular,
the point $\inc(\str(E))$ is well-defined. 
On the other hand, Lemma~\ref{maximal-angle} implies that any \emph{non}full $(n-2)$-dimensional
face of $\Tt$ is incident to at least one $\vare_n$-small simplex.
Note that it is possible to construct triangulations containing full
$(n-2)$-dimensional faces incident to no $\vare_n$-small simplices.
In this case, the map $\str$ is locally a branched covering (whose degree grows with $\varepsilon_n^{-1}$). See Figure~\ref{degree:fig} for an example where $\str$
has local degree 2 at some $(n-2)$-dimensional face.
\end{rem}

Recall that $k_n=5$ if $n=4$ and $k_n=4$ if $n\geqslant 5$. For later purposes we point out the following:

\begin{lemma}\label{stimaN}
 We have
$$
N\geqslant 5t.
$$
\end{lemma}
\begin{proof}
Recall that the number of $(n-2)$-dimensional faces of an $n$-simplex is $\frac{n(n+1)}2$.
Let $e_{\rm nf}$ be the number of $(n-2)$-dimensional faces of $\Tt$ that are \emph{not} full. Remark~\ref{rem1}
implies that $e_{\rm nf}\leqslant \frac{n(n+1)t_s}2$. Moreover, by definition,
every $(n-2)$-dimensional face of $\Tt$ that is not full is incident to at most five
$\vare_n$-big simplices of $\Tt$ (counted with multiplicities), so
$$
t \frac{n(n+1)}{2}=
(t_b+t_s) \frac{n(n+1)}{2}\leqslant N+5e_{\rm nf}+t_s\frac{n(n+1)}{2}\leqslant 
N+3t_sn(n+1).
$$
Since $t_s\leqslant \frac t{12}$, this implies 
that $N\geqslant \frac {tn(n+1)}4$, whence the conclusion since $n\geqslant 4$.
\end{proof}

We now decompose $M$ into the union of three subsets $M_1,M_2,M_3$. The first subset $M_1$ consists of $\delta_n$-balls centered at the inradii of the (straightened) full faces: 
$$M_1=\bigcup_{E\in \Full(\Tt)} B\big(\inc(\str(E)),\delta_n\big).$$
The subset $M_2$ is the union of all $\varepsilon_n$-big (straightened) simplices minus $M_1$, and $M_3$ is the union of all the $\varepsilon_n$-small simplices: 
$$M_2=\str\left(\bigcup_{\varepsilon_n-{\rm big}\ \sigma} |\sigma| \right) \setminus M_1, \qquad M_3=\str\left(\bigcup_{\varepsilon_n-{\rm small}\ \sigma} |\sigma |\right).$$
Recall from Lemma~\ref{algvol} that every point of $M$ lies in $\str(|\sigma|)$ for some simplex $\sigma$ of $\Tt$
(in fact, $\sigma$ may also be chosen to be positive, but this is not relevant here).
Therefore, we have
\begin{equation}\label{somma}
\Vol(M)\leqslant \Vol (M_1)+\Vol (M_2)+\Vol (M_3).
\end{equation}
The reason for considering these three regions is roughly the following: if there are many $\varepsilon_n$-small simplices, some volume is ``lost'' in $M_3$; on the other hand, if there are many $\varepsilon_n$-big simplices they must wind and overlap a lot along the full faces of $\Tt$ and some volume is ``lost'' in $M_1$: in all cases the volume of $M$ will be strictly smaller than $C_ntv_n$ for some constant $C_n<1$.

Let us estimate $\Vol (M_i)$, $i=1,2,3$. We set $\eta_n=\Vol (B(p,\delta_n))$, 
where $p$ is any point of $\matH^n$.
We have of course
\begin{equation}\label{m1}
\Vol (M_1)\leqslant e_{\rm f} \eta_n.
\end{equation} 

Let now $\sigma$ be an $\vare_n$-big simplex of $\Tt$, and let $\nu$ be the number
of $(n-2)$-dimensional faces of $\sigma$ (considered as an abstract $n$-simplex)
which project into a full $(n-2)$-dimensional face of $\Tt$.
If $\widetilde{\sigma}^\stra$ is a lift of $\sigma^\stra$ to $\matH^n$, 
then by Lemma~\ref{palle} the hyperbolic balls of radius $\delta_n$
centered in the incenters of the $(n-2)$-dimensional faces of $|\widetilde{\sigma}^\stra|$
are pairwise disjoint. Moreover, each of these balls does not intersect
any $(n-1)$-dimensional face of $|\widetilde\sigma^\stra |$ that does not contain its center.
Together with Lemma~\ref{maximal-angle}, this implies that the volume of
$\str(|\sigma|)\setminus M_1$ is at most
$$
v_n- \nu \eta_n \frac{1+a_n}{k_n+1}.
$$
Summing up over all the $\vare_n$-big simplices of $\Tt$ we get
\begin{equation}\label{m2}
\Vol (M_2)
\leqslant   t_b  v_n -\eta_n (1+a_n) \frac{N}{k_n+1}.
\end{equation}
Finally, we obviously have
\begin{equation}\label{m3}
 \Vol(M_3)\leqslant t_s v_n.
\end{equation}
Putting together the inequalitites~\eqref{somma}, \eqref{m1}, \eqref{m2}, \eqref{m3} we get the inequality
\begin{equation}\label{stima1}
 \Vol (M)\leqslant t v_n+ \eta_n \left(e_{\rm f}-(1+a_n) \frac{N}{k_n+1}\right).
\end{equation}

We now conclude by considering separately the cases $e_{\rm f} \leqslant \frac t2$ and $e_{\rm f} \geqslant \frac t2$.

\begin{lemma}\label{stima2}
Suppose that $t_s\leqslant \frac t{12}$ and $e_{\rm f}\leqslant \frac t2$. Then
$$
\Vol(M)\leqslant t v_n \left( 1-\frac{\eta_n}{3v_n}\right).
$$
\end{lemma}
\begin{proof}
Recall that $k_n+1\leqslant 6$ and 
$N\geqslant 5t$ (see Lemma~\ref{stimaN}), so our estimate~\eqref{stima1} yields
$$
\Vol(M)\leqslant t v_n+\eta_n\left(e_{\rm f}-\frac{5t}{6}\right)\leqslant
t v_n-\eta_n\frac{t}{3}=
t v_n \left(1-\frac{\eta_n}{3 v_n}\right).
$$
\end{proof}

\begin{lemma}\label{stima3}
  Suppose that $t_s\leqslant \frac t{12}$ and $e_{\rm f}\geqslant \frac t2$. Then
$$
\Vol(M)\leqslant t v_n \left( 1-\frac{a_n \eta_n}{2 v_n}\right).
$$
\end{lemma}
\begin{proof}
 Recall that every full $(n-2)$-dimensional face is incident to at least $k_n+1$ $\vare_n$-big simplices
of $\Tt$, so $N\geqslant (k_n+1)e_{\rm f}$. Plugging this inequality into~\eqref{stima1} we get
$$
\Vol(M)\leqslant tv_n+\eta_n(e_{\rm f}-(1+a_n)e_{\rm f})=tv_n-a_n\eta_ne_{\rm f}\leqslant 
tv_n-\frac{a_n\eta_nt}{2}.
$$
\end{proof}

We can summarize the results proved in Lemmas~\ref{stima1a}, \ref{stima2}, and 
\ref{stima3} in the following
statement, which provides a quantitative version of Theorem~\ref{bah}.

\begin{teo}
Let $M$ be a closed orientable hyperbolic manifold of dimension $n\geqslant 4$, and let
$$
C_n=\max \left\{1-\frac{\vare_n}{12}, 1-\frac{\eta_n}{3v_n}, 1-\frac{a_n\eta_n}{2v_n}\right\} < 1.$$
Then
$$
\Vol (M)\leqslant C_n v_n \sigma(M).
$$
\end{teo}

\section{Stable complexity} \label{complexity:section}
As anticipated in the introduction, by replacing triangulations with spines we get another characteristic number $c_\infty$ which equals $\sigma_\infty$ on any irreducible 3-manifold with infinite fundamental group, but which is better-behaved and closer to the simplicial volume in many cases (see \emph{e.g.}~Propositions~\ref{surface2} and \ref{elliptic}). We define here the characteristic number $c_\infty$ and prove some basic properties.

\subsection{Complexity}
The complexity $c(M)$ of a compact manifold $M$ was defined by Matveev \cite{Mat} in dimension 3 and generalized by the last author in all dimensions \cite{Mar}. We recall briefly its definition.

Let $\Delta = \Delta_{n+1}$ be the $(n+1)$-simplex and $\Pi^n$ be the cone over the $(n-1)$-skeleton of $\Delta$. The polyhedron $\Pi^n_k = \Pi^{n-k}\times D^k$ has a \emph{center} $c=(d,0)$ with $d\in\Pi^{n-k}$ being the center of the cone. A compact $(n-1)$-dimensional polyhedron $X$ is \emph{simple} if every point $x$ of $X$ has a star neighborhood PL-homeomorphic to $\Pi^n_k$, via a homeomorphism that sends $x$ to $c$.

\begin{figure}
 \begin{center}
  \includegraphics[width = 9 cm]{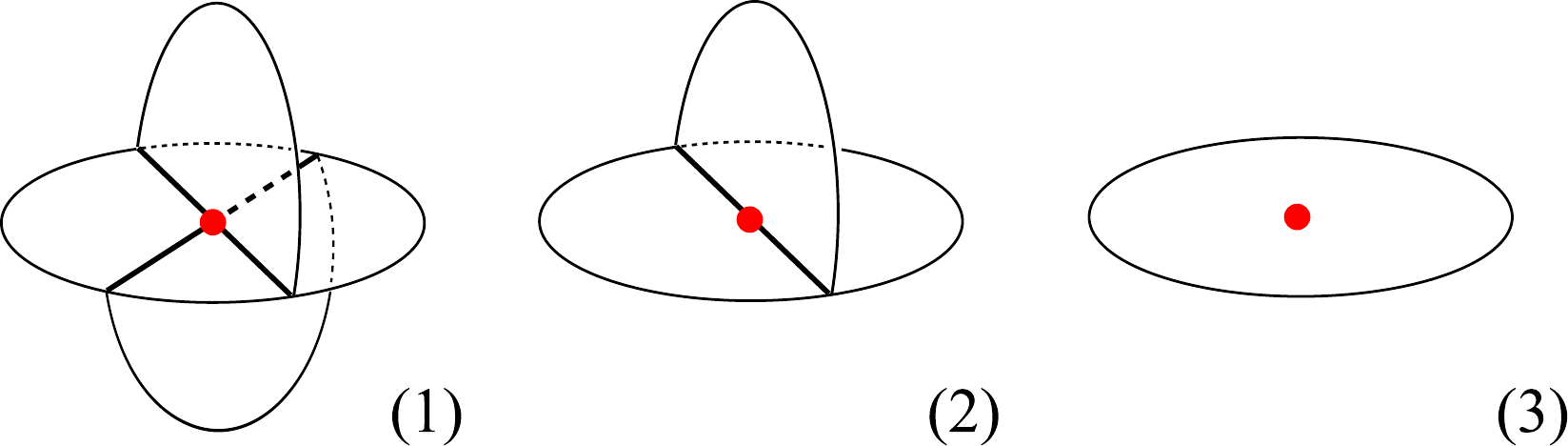}
 \end{center}
 \nota{Neighborhoods of points in a simple polyhedron.}
 \label{models:fig}
\end{figure}

In a simple 2-dimensional polyhedron every point has a neighborhood of one of the three types shown in Fig.~\ref{models:fig}. Points of type (1) are called \emph{vertices}. The points of type (2) and (3) form respectively some manifolds of dimension 1 and 2: their connected components are called respectively \emph{edges} and \emph{regions}. Note that an edge can be a circle and a region can be an arbitrary (connected) surface. A simple $n$-dimensional polyhedron is stratified similarly.

Let $M$ be a compact $n$-manifold, possibly with boundary. A subpolyhedron $X\subset \interior M$ is a \emph{spine} of $M$ 
if $M\setminus X$ consists of an open collar of $\partial M$ and some (possibly none) open balls. 

\begin{defn}
The \emph{complexity} $c(M)$ of $M$ is the minimal number of vertices in a simple spine for $M$.
\end{defn}

The following facts, already proved in \cite{Mat, Mar}, are immediate.

\begin{teo} \label{immediate:teo}
The following inequalities hold:
\begin{itemize}
\item $c(M)\leqslant \sigma(M)$ for any closed manifold $M$,  
\item $c(M)\leqslant d\cdot c(N)$ for any finite covering $M\stackrel d\to N$ of compact manifolds.
\end{itemize}
\end{teo}
\begin{proof}
By dualizing a triangulation $\calT$ of $M$ we get a simple spine of $M$ with one vertex at the barycenter of each simplex of $\calT$, hence $c(M)\leqslant \sigma (M)$.  The preimage of a simple spine of $N$ along the covering map is a simple spine of $M$ with $d$ vertices lying above each vertex of $N$, hence $c(M) \leqslant d\cdot c (N)$.
\end{proof}

We summarize the properties of $c$ in dimension 3 that we will need below. If $F\subset \interior{M}$ is a closed surface in the interior of a compact 3-manifold $M$ we denote by $M/\!/F$ the manifold $M$ with an open tubular neighborhood of $F$ removed. 

\begin{teo} [Matveev, \cite{Mat}]\label{matveev:teo}
The complexity $c$ of compact orientable 3-manifolds satisfies the following properties:
\begin{itemize}
\item $c(M) = \sigma(M)$ for any closed irreducible 3-manifold $M$ distinct from $S^3$, $\matRP^3$, and $L(3,1)$,
\item $c(M\#N) = c(M)+c(N)$ for any compact 3-manifolds $M$ and $N$,
\item  $c(M/\!/F)\leqslant c(M)$ for any irreducible compact 3-manifold $M$ and any incompressible closed surface $F\subset \interior{M}$.
\end{itemize}
\end{teo}

\subsection{Stable complexity}
Theorem \ref{immediate:teo} says that $c(M)\leqslant d\cdot c(N)$ for any finite covering $M\stackrel d\to N$ between compact manifolds \cite{Mat, Mar}. We can then mimic the construction of $\sigma_\infty$ and define the \emph{stable complexity} $c_\infty(M)$ of a compact manifold $M$ as
$$c_\infty (M) = \inf_{\widetilde M \stackrel d\to M} \left\{\frac{c(\widetilde M)}d\right\}.$$
The stable complexity is of course a characteristic number and we get 
$$c_\infty(M) \leqslant \sigma_\infty (M)$$ 
for any closed manifold $M$ by Theorem \ref{immediate:teo}.

The following result refines Proposition~\ref{easy:prop}.
\begin{prop} \label{smaller:prop}
Let $M$ be a closed $n$-manifold and suppose that $\pi_1(M)$ is virtually torsion-free
(this condition is automatically satisfied if $n=2$ and if $n=3$ thanks to geometrization).
Then
$$\|M\| \leqslant c_\infty (M) \leqslant \sigma_\infty (M).$$
\end{prop}
\begin{proof}
As we have just said, the right inequality is in fact true for any closed $M$. Concerning the left inequality, we have $\|N\| \leqslant c(N)$ for any closed manifold $N$ with virtually torsion-free fundamental group \cite{Mar}. Since this group-theoretical property extends to every finite index subgroup of $\pi_1(M)$, for every degree-$d$ covering $\widetilde M \to M$ of $M$ we get 
$$\|M\| = \frac{\|\widetilde M \|}d \leqslant \frac{c(\widetilde M)}d$$
and therefore $\|M\|\leqslant c_\infty(M)$. 

If $N$ is a closed $3$-manifold,
the inequality $\|N\| \leqslant c(N)$ 
can be proved directly (without geometrization) building on these facts:
\begin{itemize}
\item both $c$ and $\|\cdot \|$ are additive on connected sums \cite{Mat, Gro}; 
\item if $M\in \{S^3, \matRP^3, S^2\times S^1, L(3,1)\}$ then $c(M)=0$ \cite{Mat};
\item if $M$ is irreducible and not in the above list then $c(M)=\sigma(M)$ (\cite{Mat}, see Theorem~\ref{matveev:teo}).
\end{itemize}
\end{proof}

Turning to dimension 3, we will prove below an appropriate version of Theorem \ref{matveev:teo} for $c_\infty$. First of all, the characteristic numbers $c_\infty$ and $\sigma_\infty$ coincide on the 3-manifolds we are mostly interested in:

\begin{prop} Let $M$ be a closed irreducible 3-manifold with $|\pi_1(M)|=\infty$. Then $c_\infty(M) = \sigma_\infty (M)$.
\end{prop}
\begin{proof}
Every finite-index covering $N$ of $M$ is irreducible with $|\pi_1(N)|=\infty$ and hence $c(N)=\sigma(N)$. Therefore $c_\infty(M) = \sigma_\infty(M)$.
\end{proof}

We will show in the next section that $c_\infty$ is also additive on connected sums and monotonic with respect to cutting along incompressible surfaces. More than that, we will show that $c_\infty$ is also additive on JSJ decompositions. Note that $c$ is certainly \emph{not} additive on JSJ decompositions, since there are only finitely many irreducible 3-manifolds of any given complexity $c$ (because $c=\sigma$ there), whereas infinitely many 3-manifolds can share the ``same'' JSJ decomposition (in the weak sense that they share the same geometric blocks, but assembled via different maps).

\subsection{Surfaces and elliptic manifolds}
The following propositions describe
examples where the stable complexity is equal to the simplicial volume and strictly smaller than 
the stable $\Delta$-complexity.

\begin{prop}\label{surface2}
If $S$ is a compact surface then $c_\infty (S) = \| S\|= 2\chi_-(S)$,
where $\chi_-(S) = \min\{-\chi(S), 0\}$.
Therefore, if $S$ is closed 
we have $\sigma_\infty(S)>c_\infty (S)$ if $\chi(S)>0$ and 
$\sigma_\infty(S)=c_\infty(S)$ if $\chi(S)\leqslant 0$.
\end{prop}
\begin{proof}
Let us first recall that the equality $\| S\|=2\chi_-(S)$ (which was stated
in Propositon~\ref{surface} for closed surfaces) also holds for surfaces with boundary.
In fact, if $S$ is a disk or an annulus, then the pair
$(S, \partial S)$ admits a self-map of degree bigger than
one, so $\| S\|=\chi_-(S)=0$. If $S$ is
a M\"obius strip, then $S$ is covered by the annulus, so again
$\| S\|=\chi_-(S)=0$. In the remaining cases, the interior of
$S$ admits a complete finite-volume hyperbolic structure, so we may apply
Theorem~\ref{prop:teo} to get
$$
\|S\|=\frac{{\rm Area}({\rm int}(S))}{v_2}=\frac{2\pi|\chi(S)|}{\pi}=2\chi_-(S) ,
$$
where $\vol({\rm int}(S))=2\pi|\chi(S)|$ by Gauss-Bonnet Theorem,
and $v_2=\pi$ since the maximal area of hyperbolic triangles is equal to 
$\pi$.

Let us now come to the statement of the proposition.
If $S$ is either $S^2$, $\matRP^2$, an annulus, or a M\"obius strip,
then $S$ has a spine without vertices (a circle) and hence $c(S)=0$. Every other surface $S$ with non-empty boundary has a simple (\emph{i.e.}~trivalent) spine with $2\chi_-(S)$ vertices, and hence 
$$\|S\|\leqslant c_\infty(S) \leqslant c(S) \leqslant 2\chi_-(S) = \|S\| .$$
This
proves 
the first statement
when $S$ has non-empty boundary or $\chi(S)>0$. 
If $S$ is closed with $\chi(S)\leqslant 0$ then by Proposition~\ref{easy:prop}
we have
$$\| S\|\leqslant c_\infty (S)\leqslant \sigma_\infty(S)= \| S\|=2\chi_-(S),$$
whence the conclusion.
\end{proof}

\begin{prop}\label{elliptic}
If $M$ is an elliptic $n$-manifold then $c_\infty(M)=\| M\|=0$ and $\sigma_\infty(M)>0$.
\end{prop}
\begin{proof}
For every $n\geqslant 1$
we have $c(S^n)=0$ \cite{Mat, Mar}
and $\| S^n\|=0$, since $S^n$ admits a self-map of degree bigger than one.
Since every elliptic manifold is covered by $S^n$, we get 
$c_\infty(M)=\| M\|=0$.
On the other hand $\sigma(M)>0$ for every manifold $M$ and hence $\sigma_\infty(M)>0$ whenever $M$ has finite fundamental group, and hence only finitely many coverings.
\end{proof}

\section{Three-manifolds} \label{Three:section}

We study here the stable complexity $c_\infty$ of 3-manifolds. We first show that $c_\infty$ is additive on connected sums and JSJ decompositions: while additivity on connected sums is easy, to prove additivity on JSJ decompositions we make an essential use of a couple of lemmas established by Hamilton \cite{Ham}. As a corollary, we compute $c_\infty$ on any irreducible 3-manifold whose JSJ decomposition consists of Seifert pieces and hyperbolic manifolds commensurable with the figure-eight knot complement. 

We end this section by exhibiting a sequence of closed hyperbolic 3-manifolds $M_i$ (with bounded volume) for which the ratio between $c_\infty(M_i)$ and $\|M_i\|$ tend to one.

\subsection{Minimizing sequences and disconnected coverings.}
We define the following natural notion.
\begin{defn}
Let $M$ be a compact manifold. A \emph{minimizing sequence} of 
coverings $f_i\colon M_i\stackrel{d_i}{\to} M$ is a sequence such that $\frac{c(M_i)}{d_i} \to c_\infty(M)$. 
\end{defn}
Of course every manifold $M$ has a minimizing sequence. Given a minimizing sequence $f_i\colon M_i\stackrel{d_i}{\to} M$, we can replace each $M_i$ with any manifold $N_i$ covering $M_i$ and we still get a minimizing sequence. In particular, if $M$ is a 3-manifold with $|\pi_1(M)|=\infty$ we can always take a minimizing sequence such that $d_i\to \infty$ because $\pi_1(M)$ is residually finite \cite{Hem}.

Let $\eta(M)$ be an invariant which is submultiplicative under finite coverings, like $c(M)$, $\sigma(M)$, or $\|M\|^\matZ$. We have defined in this paper the \emph{stable} version $\eta_\infty(M)$ of $\eta(M)$ by taking the infimum of $\frac{\eta(N)}d$ among all finite coverings $N\stackrel d\to M$. We have implicitly assumed in this definition that both $M$ and $N$ are connected, as this hypothesis is typically embodied in the definition of ``covering''. If we discard this hypothesis, thus allowing both $M$ and $N$ to be disconnected, we actually get the same stable function $\eta_\infty$. 

More precisely, we define a \emph{(possibly disconnected) degree-$d$ covering} as a map $p\colon M\to N$ between (possibly disconnected) topological spaces where every point in $N$ is contained in some open set $U$ such that $p^{-1}(U) = \cup_{i=1}^d U_i$ and $p|_{U_i}\colon U_i \to U$ is a homeomorphism. We re-define $\eta_\infty(M)$ for any (possibly disconnected) manifold $M$ as the infimum of $\frac{\eta(M)}d$ over all (possibly disconnected) degree-$d$ coverings of $M$. It is easy to verify that this slightly modified definition of $\eta_\infty(M)$ coincides on a connected manifold $M$ with the one we have introduced beforehand using only connected coverings, and that we get an additive function $\eta_\infty(\sqcup_{i\in I} M_i) = \sum_{i\in I} \eta_\infty(M_i)$ on the connected components of disconnected manifolds.

In this section (and nowhere else) we allow implicitly all converings to be disconnected: this is a natural framework when one cuts a 3-manifold along surfaces, and might get a disconnected 3-manifold as a result.

\subsection{Connected sums and incompressible surfaces}
Additivity on connected sums easily lifts from $c$ to $c_\infty$. We subdivide the proof in two steps.

\begin{prop} Let $M$ be a 3-manifold and $S\subset\interior M$ a 2-sphere. We have $c_\infty (M/\!/S) = c_\infty(M)$.
\end{prop}
\begin{proof}
We know \cite{Mat} that if $N$ is a 3-manifold and $S\subset\interior N$ is a sphere then $c(N/\!/S) = c(N)$, so the same result for $c_\infty$ follows easily. 

If $p\colon \widetilde M \to M$ is a covering, the preimage $\widetilde S = p^{-1}(S)$ is a union of spheres, and hence $c(\widetilde M/\!/\widetilde S) = c(\widetilde M)$. Every covering $\widetilde M \stackrel d\to M$ induces a covering $\widetilde M  /\!/ \widetilde S \stackrel d\to M/\!/S$ with $c(\widetilde M) = c(\widetilde M /\!/ \widetilde S)$, hence $c_\infty(M/\!/S)\leqslant c_\infty(M)$. Conversely, every covering $N \stackrel d \to M/\!/S$ gives rise to a covering $N' \stackrel d \to M$, where $N'$ is obtained from $N$ by gluing the $2d$ boundary spheres in pairs. In particular $c(N') = c(N)$ and hence we also get $c_\infty(M/\!/S)\geqslant c_\infty(M)$.
\end{proof}

\begin{cor} Let $M,N$ be any compact 3-manifolds. We have 
$$c_\infty(M\#N) = c_\infty(M)+c_\infty(N).$$
\end{cor}
\begin{proof}
Cutting and glueing along 2-spheres does not vary $c_\infty$. Capping a boundary 2-sphere with a 3-disc $D^3$ also does not modify $c_\infty$ since $c_\infty(D^3) = c(D^3) = 0$.
\end{proof}
Another property which lifts easily from $c$ to $c_\infty$ is monotonicity under the operation of cutting along incompressible surfaces. 

\begin{prop} Let $S\subset \interior{M}$ be an incompressible surface in an irreducible 3-manifold $M$. We have
$$c_\infty(M/\!/S) \leqslant c_\infty(M).$$
\end{prop}
\begin{proof}
If $p\colon \widetilde M \to M$ is a covering, the manifold $\widetilde M$ is irreducible and the pre-image $\widetilde S = p^{-1}(S)$ of $S$ is a (possibly disconnected) incompressible surface in $\widetilde M$. Therefore $c(\widetilde M /\!/ \widetilde S)\leqslant c(\widetilde M)$ by Theorem \ref{matveev:teo}, and
$c_\infty(M/\!/S)\leqslant c_\infty(M)$.
\end{proof}

When the incompressible surface is a torus, we actually get an
equality. To prove this non-trivial fact (which does not hold for $c$ and heavily depends on geometrization) 
we will need to construct appropriate
coverings of irreducible 3-manifolds, using some techniques introduced
by Hempel in his proof that the fundamental group of an irreducible
3-manifold is residually finite \cite{Hem} and further developed in a
recent paper by E. Hamilton \cite{Ham}.

\subsection{Characteristic coverings}
Recall that a \emph{characteristic subgroup} of a group $G$ is a subgroup $H<G$ which is invariant by any automorphism of $G$. For a natural number $x\in \matN$, the \emph{$x$-characteristic} subgroup of $\matZ\times \matZ$ is the subgroup $x(\matZ\times\matZ)$ generated by $(x,0)$ and $(0,x)$. It has index $x^2$ if $x>0$ and $\infty$ if $x=0$. The characteristic subgroups of $\matZ\times \matZ$ are precisely the $x$-characteristic subgroups with $x\in \matN$. It is easy to prove that a subgroup of $\matZ\times\matZ$ of index $x$ contains the $x$-characteristic subgroup.

A covering $p\colon \widetilde T\to T$ of tori is called \emph{$x$-characteristic} if $p_*(\pi_1(\widetilde T))$ is the $x$-characteristic subgroup of $\pi_1(T)\Isom \matZ\times \matZ$. A covering $p\colon \widetilde M \to M$ of 3-manifolds bounded by tori is \emph{$x$-characteristic} if the restriction of $p$ to each boundary component of $\widetilde M$ is $x$-characteristic.

Lemmas 5 and 6 from \cite{Ham} state the following.

\begin{lemma}[E. Hamilton] \label{Hamilton:lemma}
Let $M_1,\ldots, M_n$ be a finite collection of compact, orientable 3-manifolds with boundary whose interiors admit complete hyperbolic structures of finite volume. Let $m$ be a positive integer. Then there exist a positive integer $x$ and finite-index normal subgroups $K_i\triangleleft \pi_1(M_i)$ such that $K_i\cap \pi_1(T_{ij})$ is the characteristic subgroup of index $(mx)^2$ in $\pi_1(T_{ij})$, for each component $T_{ij}$ of $\partial M_i$. Hence the covering of $M_i$ corresponding to $K_i$ is $(mx)$-characteristic.
\end{lemma}

\begin{lemma}[E. Hamilton] \label{Hamilton:2:lemma}
Let $M$ be a compact, orientable Seifert fibered space with non-empty, incompressible boundary. Then there exists a positive integer $v$ such that for each multiple $m$ of $v$ there is a finite $m$-characteristic covering space $M_m$ of $M$.
\end{lemma}

We will use these lemmas to prove the following result, which concerns this question: given one covering on each piece of the JSJ decomposition of an irreducible 3-manifold $M$, can we glue them together to a covering of $M$? The answer is of course negative in general, since there is no way to glue arbitrary coverings which behave very differently along the tori of the JSJ decomposition; however Lemmas \ref{Hamilton:lemma} and \ref{Hamilton:2:lemma} can be used to replace the given coverings with some bigger $x$-characteristic coverings, and (as noted by Hempel \cite{Hem}) such coverings can indeed be glued together (but one needs to take multiple copies of each covering to glue everything properly). 

\begin{prop} \label{glue:prop}
Let an irreducible orientable 3-manifold $M$ with (possibly empty) boundary consisting of tori decompose along its JSJ decomposition into some pieces $M_1,\ldots, M_h$. Let $p_i\colon\widetilde{M_i}\to M_i$ be a finite covering for every $i$. There exist a natural number $n$, a finite covering $q_i\colon N_i \to \widetilde{M_i}$ for every $i$ and a finite covering $p\colon N \to M$ such that $p^{-1}(M_i)$ consists of copies of $N_i$ covering $M_i$ along $p_i\circ q_i$. Moreover each $p_i\circ q_i$ is $n$-characteristic.
\end{prop}
\begin{proof}
Up to taking a bigger covering we can suppose that $p_i$ is regular for every $i$. The pre-image of a boundary torus $T_{ij}\subset\partial M_i$ consists of finitely many tori $T_{ij}^1, \ldots, T_{ij}^l$, and the restriction of $p_i$ to each $T_{ij}^k$, $k=1,\ldots,l$, gives isomorphic coverings (because $p_i$ is regular). In particular $H_{ij} = (p_i)_*(\pi_1(T_{ij}^k))$ is a subgroup of $\pi_1(T_{ij})$ which does not depend on $k$. Let $d_{ij}$ be the index of $H_{ij}$ in $\pi_1(T_{ij})$.

By geometrization every $M_i$ is either hyperbolic or Seifert. For every Seifert block $M_i$ there is some integer $v_i$ such that the conclusion of Lemma \ref{Hamilton:2:lemma} applies. Let now $m$ be the least common multiple of all integers $d_{ij}$ and $v_i$. Let us apply Lemma \ref{Hamilton:lemma} to the hyperbolic blocks of the JSJ decomposition of $M$: there is an integer $x$ such that every hyperbolic block $M_i$ has an $(mx)$-characteristic covering. By Lemma \ref{Hamilton:2:lemma} every Seifert block $M_i$ also has an $(mx)$-characteristic covering.

Therefore, every block $M_i$ has an $(mx)$-characteristic covering, determined by some subgroup $K_i<\pi_1(M_i)$ which intersects every $\pi_1(T_{ij})$ in its $(mx)$-characteristic subgroup.
Recall that our original covering $p\colon \widetilde{M_i}\to M_i$ is determined by some other subgroup $H_i<\pi_1(M_i)$ intersecting every $\pi_1(T_{ij})$ in a subgroup $H_{ij}$ of some index $d_{ij}$. A subgroup of index $d_{ij}$ contains the $(d_{ij})$-characteristic subgroup and hence the $(mx)$-characteristic subgroup since $d_{ij}$ divides $mx$. Therefore $K_i\cap H_i$ also intersects every $\pi_1(T_{ij})$ in its $(mx)$-characteristic subgroup, and hence it induces a $(mx)$-characteristic covering $N_i\to \widetilde{M_i} \to M_i$.

Summing up, we have shown that every covering $\widetilde {M_i}\to M_i$ has a bigger $(mx)$-characteristic covering $N_i\to\widetilde {M_i}\to M_i$, where the constant $mx$ is fixed. Hempel proved \cite{Hem} that $K$-characteristic coverings (with fixed $K$) can be glued together. Namely, there is a finite covering $N\to M$ such that its restriction to $M_i$ consists of finite copies of the covering $N_i\to M_i$.
\end{proof}

\begin{cor} \label{glue:cor}
Let $M$ be an irreducible orientable 3-manifold. For every integer $n_0$ there is a bigger integer $n>n_0$ and a covering $p\colon N\to M$ whose restriction over any torus of the JSJ decomposition of $M$ is a disjoint union of $n$-characteristic coverings.
\end{cor}
\begin{proof}
Take a block $M_1$ of the JSJ decomposition of $M$. Thanks to geometrization, the fundamental group $\pi_1(M_1)$ is residually finite, hence there is a covering $\widetilde M_1 \to M_1$ which restricts on some boundary torus of $\widetilde M_1$ to a covering of degree bigger than $n_0^2$. Apply Proposition \ref{glue:prop} to this covering: the result is an $n$-characteristic covering $N\to M$ with $n > n_0$.
\end{proof}

\subsection{JSJ decompositions}

We will prove below that $c_\infty$ is additive on JSJ decompositions. We start by proving the following. 

\begin{lemma} \label{tori:lemma}
Let an irreducible orientable 3-manifold $M$ with (possibly empty) boundary consisting of tori decompose along its JSJ decomposition into some pieces $M_1,\ldots, M_h$. We have
$$c_\infty (M) \leqslant c(M_1)+\ldots + c(M_h).$$
\end{lemma}
\begin{proof}
Given some simple spines $P_1, \ldots, P_h$ for $M_1, \ldots, M_h$, it is easy to construct a simple spine $Q$ for $M$. Set $P = P_1\sqcup \ldots \sqcup P_h$. Recall that $M_i\setminus P_i$ consists of an open collar of $\partial M_i$ plus maybe some open balls. Therefore $M\setminus P$ consists of one product neighborhood $T\times (-1,1)$ of each torus $T$ of the decomposition, plus maybe some open balls. To build a spine for $M$ it suffices to choose a simple spine $Y$ for $T$ (\emph{i.e.}~$Y$ is a 1-dimensional polyhedron with only 3-valent vertices and $T\setminus Y$ consists of open discs) and add to $P$ one product $Y\times (-1,1)$ inside each such product neighborhood $T\times (-1,1)$. If $Y\subset T$ is in generic position, the resulting polyhedron $Q$ is still simple. Now $M\setminus Q$ consists of open balls only: those that were in $M\setminus P$, plus one for each torus of the decomposition. Therefore $Q$ is a spine for $M$.

\begin{figure}
\begin{center}
\includegraphics[width = 12.5 cm] {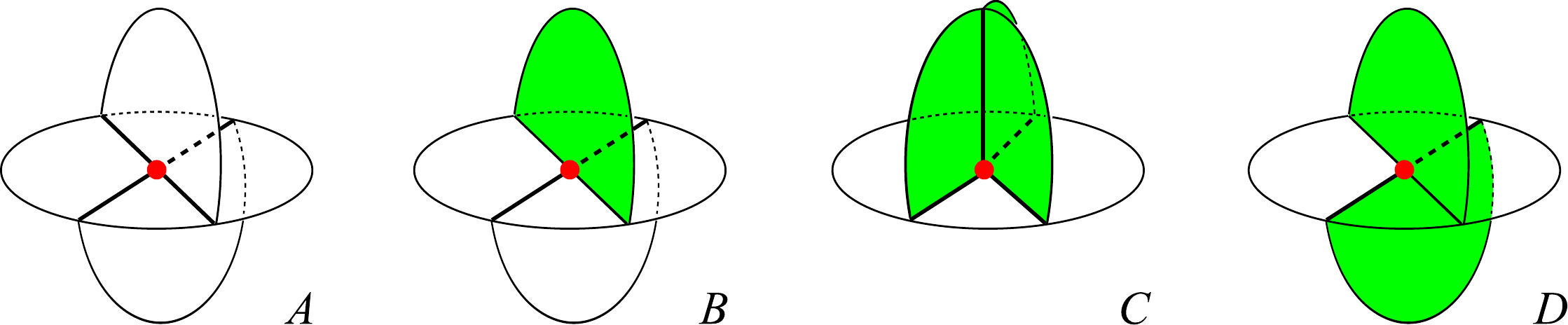}
\nota{We colour in green the regions of the inserted portions $Y\times (-1,1)$. There are four types of vertices $A$, $B$, $C$, and $D$ in the spine $Q$, according to the colours of the incident regions.}
\label{types:fig}
\end{center}
\end{figure}

Colour in green the regions in the products $Y\times (-1,1)$. It is easy to check that there are now four types $A, B, C, D$ of vertices in $Q$ according to the colours of the incident regions, as shown in Fig.~\ref{types:fig}. The vertices of type $A$ are those of $P$. Let $v_A$, $v_B$, $v_C$, and $v_D$ be the number of vertices of type $A$, $B$, $C$, and $D$ in $Q$. Consider one inserted piece $Y\times (-1,1) \subset T\times (-1,1)$ inside a collar separating two (possibly coinciding) polyhedra $P_i$ and $P_j$. Pull back the cellularization of $P_i$ on $T$ via the collar, as in Fig.~\ref{tori:fig}: the four types of vertices are also shown in the figure. 

\begin{figure}
\begin{center}
\includegraphics[width = 9 cm] {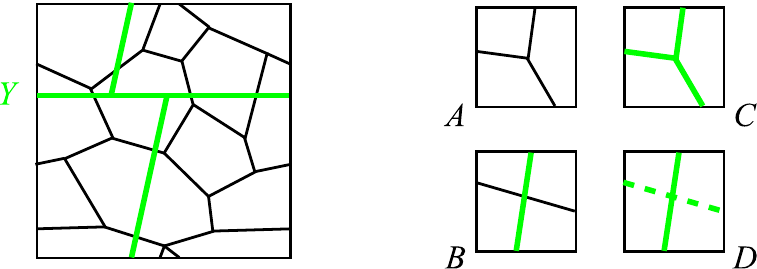}
\nota{The cellularization of $T$ induced by the collar map $T\to P_i$, and the spine $Y$ of $T$ coloured in green. The four types of vertices $A$, $B$, $C$, $D$.}
\label{tori:fig}
\end{center}
\end{figure}

Corollary \ref{glue:cor} ensures that for every $n_0>0$ there is a natural number $n>n_0$ and a covering $p\colon N\stackrel{d}\to M$ whose restriction over each torus $T$ of the JSJ decomposition is a disjoint union of some $h$ distinct $n$-characteristic coverings. We thus have $d = hn^2$. The pre-image $\widetilde Q = p^{-1}(Q) \subset N$ is a simple spine of $N$, and we give each region of $\widetilde Q$ the same colour of its image in $Q$. We thus get $dv_A, dv_B, dv_C,$ and $dv_D$ vertices of type $A$, $B$, $C$, $D$ respectively.

\begin{figure}
\begin{center}
\includegraphics[width = 12.5 cm] {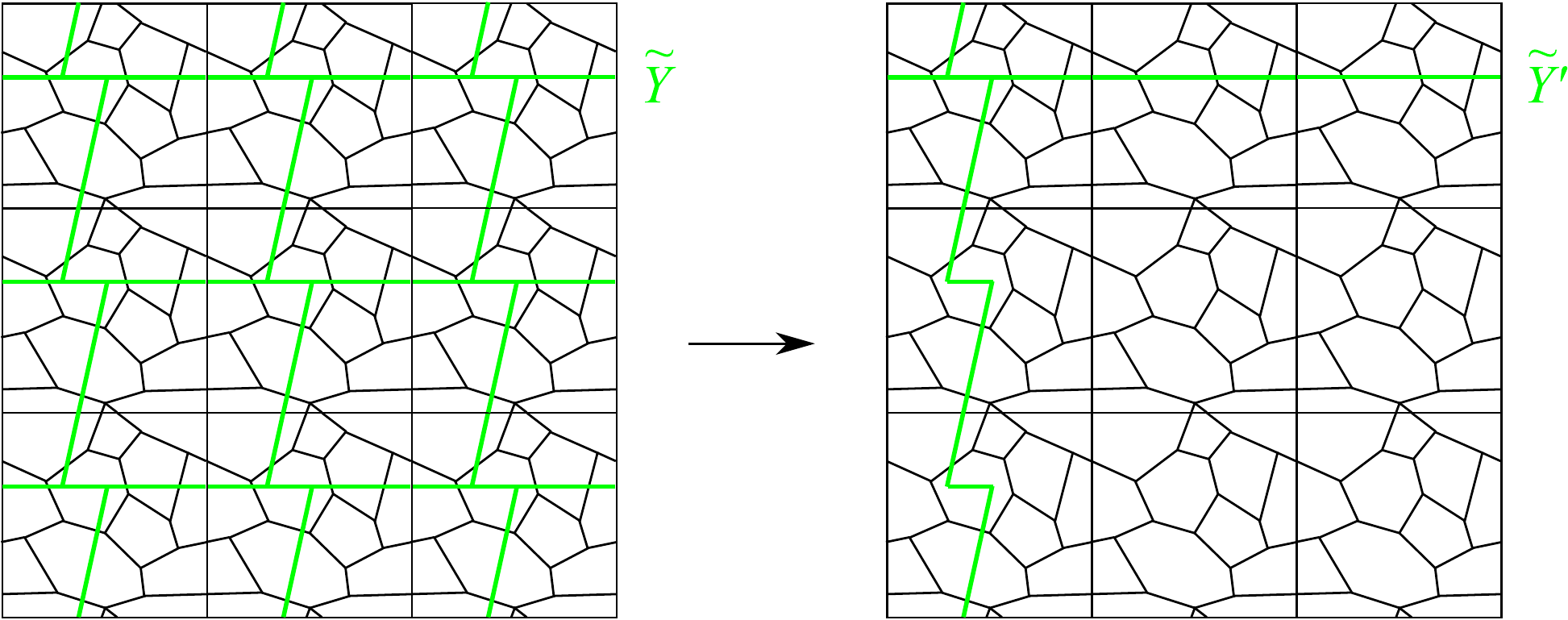}
\nota{A 3-characteristic covering $\widetilde T$ of some torus $T$ of the JSJ decomposition. The spine $Y$ lifts to the green spine $\widetilde Y$ shown in the left picture. We can eliminate most of its edges and still get a spine $\widetilde Y'$ of $\widetilde T$.}
\label{tori2:fig}
\end{center}
\end{figure}

Let $T\subset M$ be one torus of the decomposition. One component $\widetilde T$ of $p^{-1}(T)$ is shown in Fig.~\ref{tori2:fig}, containing the lifted spine $\widetilde Y$. If $T\setminus Y$ consists of one disc only as in Fig.~\ref{tori:fig}, now $\widetilde T \setminus \widetilde Y$ consists of $n^2$ discs. As shown in the figure, we can replace  $\widetilde Y$ with a simpler spine $\widetilde Y'\subset \widetilde Y \subset \widetilde T$, whose complement in $\widetilde T$ consists of only one disc. We then modify $\widetilde Q$ by substituting the product $\widetilde Y \times (0,1)$ with $\widetilde Y' \times (0,1)$. The resulting polyhedron $\widetilde Q'\subset \widetilde Q$ is still a spine of $N$, with less vertices than $\widetilde Q$. 

We estimate the number of vertices for $\widetilde Q'$. Recall that $d = hn^2$. It is clear from the picture that, after the removal of such coloured faces, the number of vertices of type $A$, $B$, $C$, $D$ is respectively not greater than $dv_A$, $2hnv_B$, $hv_C$, and $2hnv_D$. Therefore 
$$c(N) \leqslant dv_A + 2hn(v_B+v_D) + hv_C.$$ 

Suppose that in our construction we started with some spines $P_1,\ldots, P_h$ with minimal number of vertices for $M_1,\ldots, M_h$. Then $v_A$ equals $c(M_1)+\ldots + c(M_h)$ and we get
$$c_\infty (M) \leqslant \frac{c(N)}{d} \leqslant v_A + \frac{2(v_B+v_D)}{n} + \frac{v_C}{n^2}.
$$
Since for every $n_0$ there is $n>n_0$ which satisfies this inequality, we get 
$$c_\infty(M) \leqslant v_A = c(M_1) +\ldots + c(M_h).$$
\end{proof}

Finally we can prove the following.

\begin{prop} Let an irreducible orientable 3-manifold $M$ with (possibly empty) boundary consisting of tori decompose along its JSJ decomposition into some pieces $M_1,\ldots, M_h$. We have
$$c_\infty (M) = c_\infty(M_1)+\ldots + c_\infty(M_h).$$
\end{prop}
\begin{proof}
We already know that by cutting along incompressible surfaces we cannot increase the stable complexity, hence $c_\infty (M)\geqslant c_\infty(M_1)+\ldots + c_\infty(M_h)$. We need to prove the converse inequality. 

Let $p_i^j\colon M_i^j\stackrel{d_i^j}{\to} M_i$ be a minimizing sequence of coverings for $M_i$, for each $i=1,\ldots, h$. By hypothesis we have
$$\frac{c(M_i^j)}{d_i^j} \to c_\infty (M_i)$$
as $j\to \infty$, for each $i=1,\ldots, h$.

Fix $j$. By Proposition \ref{glue:prop}, up to replacing each $p_i^j$ with a bigger covering (which we still denote by $p_i^j$) we can suppose that there is a covering $p\colon M^j\stackrel{d^j}{\to} M$ which restricts on $M_i$ to some $k_i^j$ disjoint copies of $p_i^j$, for every $i$. We necessarily have $d^j = k_i^jd_i^j$. Lemma \ref{tori:lemma} impies that
$$c_\infty (M^j) \leqslant k_1^jc(M_1^j) + \ldots + k_h^jc(M_h^j).$$
We divide both expressions by $d^j$ and get
$$c_\infty(M) \leqslant \frac{c(M_1^j)}{d_1^j} + \ldots + \frac{c(M_h^j)}{d_h^j}.$$
Since $p_i^j$ are minimizing sequences for all $i$, by sending $j\to \infty$ we get
$$c_\infty(M) \leqslant c_\infty(M_1) + \ldots + c_\infty(M_h).$$
\end{proof}

\subsection{Seifert manifolds}
We have proved that the stable complexity of an irreducible 3-manifold is the sum of the stable complexity of the pieces in its JSJ decomposition. We can therefore concentrate our attention to Seifert and hyperbolic manifolds. 

\begin{prop} Let $M$ be a compact Seifert manifold, with or withour boundary. We have $c_\infty(M)=0$.
\end{prop}
\begin{proof}
A Seifert manifold has a finite covering $M$ which is an $S^1$-bundle over an orientable surface $\Sigma$ with some Euler number $e\geqslant 0$. If the manifold has boundary then $e=0$ and the bundle is a product $\Sigma \times S^1$. Since this manifold covers itself with arbitrarily high degree, it clearly has stable complexity zero. If the manifold $M$ is closed, we denote it by $(\Sigma, e)$. Its complexity is at most \cite{MP}:
$$c(\Sigma, e)\leqslant \max \{0, e-1+\chi(\Sigma)\} - 6(\chi(\Sigma)-1) \leqslant e+6\chi_-(\Sigma)+6.$$
A degree-$d$ covering of surfaces $\widetilde\Sigma \stackrel d\to \Sigma$ induces a covering $(\widetilde \Sigma, de) \stackrel d\to (\Sigma, e)$. By unwrapping the fiber we can construct another covering $(\widetilde \Sigma, e) \stackrel d\to (\widetilde\Sigma, de)$ and by composing them we get
$$(\widetilde \Sigma, e) \stackrel{d^2} \to (\Sigma, e).$$
Therefore 
$$c_\infty(\Sigma,e) \leqslant \frac{c(\widetilde\Sigma,e)}{d^2} \leqslant \frac{e+6\chi_-(\widetilde \Sigma) + 6}{d^2} = \frac{e+6d\chi_-(\Sigma) + 6}{d^2} \to 0$$
as $d\to \infty$.
\end{proof}

\begin{cor} A graph manifold has stable complexity zero. \end{cor}

\subsection{Hyperbolic $3$-manifolds}
We can now turn to hyperbolic 3-manifolds. Note that, since $c_\infty$ is a characteristic number, once we know the stable complexity of a manifold we also know the stable complexity of any manifold in its commensurability class. We start by extending Proposition \ref{smaller:prop} to the cusped case.

\begin{prop} Let $M$ be a compact 3-manifold, whose interior admits a complete hyperbolic structure with finite volume. We have
$$\|M\| \leqslant c_\infty(M).$$
\end{prop}
\begin{proof}
The closed case was considered in Proposition \ref{smaller:prop}, so we suppose $M$ has boundary.
As shown by Matveev \cite{Mat:book}, if $N$ is  
 a compact 3-manifold whose interior admits a complete hyperbolic structure with finite volume, then
the complexity $c(N)$ equals the minimal number of tetrahedra in an ideal triangulation of $N$. 
Moreover, by straightening the simplices of such an ideal triangulation it is easily seen that
$\vol (N)\leqslant v_n c(N)$, whence $\| N\| \leqslant c(N)$ by Theorem~\ref{prop:teo}.
Therefore $\|N\|\leqslant c(N)$ for every covering $N$ of $M$, and we conclude as in Proposition \ref{smaller:prop}.
\end{proof}
We can calculate $c_\infty$ only on one (very special) commensurability class of hyperbolic cusped 3-manifolds. 

\begin{prop} If $M$ is commensurable with the figure-eight knot complement 
$$\|M\| = c_\infty(M).$$
\end{prop}
\begin{proof}
The figure-eight knot complement $N$ is obtained by gluing two ideal regular tetrahedra, each of volume $v_3$. We have $\Vol(N) = 2v_3$ and hence $\|N\| = 2$. We also have $c(N)=2$. The (very special) equality $c(N) = \|N\|$ together with $\|N\|\leqslant c_\infty(N) \leqslant c(N)$ implies $c_\infty (N) = \|N\|$. Since both $c_\infty$ and $\|\cdot \|$ are characteristic numbers, they coincide on the whole commensurability class of $N$.
\end{proof}

Finally, we can say something concerning Dehn filling.
\begin{prop} \label{filling:prop}
Let $N$ be any compact 3-manifold and $M$ be obtained from $N$ via Dehn filling one boundary torus of $N$. We have
$$c_\infty (M) \leqslant c(N).$$
\end{prop}
\begin{proof}
The proof is similar to that of Lemma \ref{tori:lemma}. Given a simple spine $P$ of $N$, it is easy to construct a spine $Q$ of $M$. Recall that $N\setminus P$ consists of a collar of the boundary plus possibly some open balls. Then $M\setminus P$ consists of a collar of the boundary (if non-empty), plus possibly some open balls, plus an open solid torus $V$ created by the Dehn filling. Let $D$ be a meridian disc of $V$, and take $Q = P \cup D$. If $D$ is generic, the resulting polyhedron $Q$ is still simple. Now $V\setminus D$ is an open ball and hence $Q$ is a spine of $M$.

As in the proof of Lemma \ref{tori:lemma}, colour the added disc $D$ in green. There are now three types $A$, $B$, and $D$ of vertices in $Q$, as shown in Fig.~\ref{types:fig}. Let $v_A$, $v_B$, and $v_D$ be the number of vertices of type $A$, $B$, and $D$. 

Since $\pi_1(M)$ is residually finite \cite{Hem}, for every $n>0$ there is an $h>0$ and a regular covering $p\colon \widetilde M \stackrel {hn}\to M$ such that $p^{-1}(V)$ consists of $h$ open solid tori $\widetilde V_1,\ldots, \widetilde V_h$, each winding $n$ times along $V$ via $p$. The covering has degree $d=hn$. The spine $Q$ of $M$ lifts to a spine $p^{-1}(Q) = \widetilde Q$ of $\widetilde M$, which contains $dv_A$, $dv_B$, and $dv_D$ vertices of type $A$, $B$, and $D$. The disc $D$ lifts to $n$ discs inside each $\widetilde V_i$. These $n$ discs subdivide $\widetilde V_i$ into $n$ open balls. We now remove from $\widetilde Q$ some $n-1$ of these $n$ discs, leaving only one disc $\widetilde D_i\subset \widetilde V_i$, whose complement in $\widetilde V_i$ is a single open ball. If we do such removals for every $i=1,\ldots, h$ we are left with a spine $\widetilde Q'\subset \widetilde Q$ with fewer vertices.

The number of vertices of type $A$, $B$, and $D$ in $\widetilde Q$ is at most $dv_A$, $hv_B$, and $hv_D$. Therefore
$$c(\widetilde M) \leqslant dv_A + h(v_B+v_D).$$
Suppose that $P$ has the minimal number of vertices for $N$. Then $v_A$ equals $c(N)$ and we get
$$c_\infty(M) \leqslant \frac {c(\widetilde M)}d \leqslant v_A + \frac{v_B+v_D}n.$$
Since this equality holds for every $n>0$ we get $c_\infty(M) \leqslant c(N)$.
\end{proof}

Note that we do not know if $c_\infty(M)$ is smaller than $c_\infty(N)$. However, this result is enough to deduce the following corollary, which implies
in turn Theorem~\ref{3:teo}, since $\sigma_\infty(M)=c_\infty(M)$ for every closed hyperbolic $3$-manifold $M$. 
\begin{cor} \label{figure-eight:cor}
Let $M_i$ be any sequence of distinct Dehn fillings on the figure-eight knot complement. We have
$$\frac{c_\infty(M_i)}{\|M_i\|} \to 1.$$
\end{cor}
\begin{proof}
Let $N$ be the figure-eight knot complement. We have $c_\infty(M_i) \leqslant c(N) = 2$ for all $i$. By Thurston's Dehn filling Theorem~\cite{Thurston} we also get $\Vol(M_i)\to \Vol(N)$ and hence $\|M_i\|\to \|N\| = 2$. Since $c_\infty(M_i)/\|M_i\|\geqslant 1$ the conclusion follows.
\end{proof}

\section{Concluding remarks}\label{futuro:sec}

In this section we show how our arguments
can be adapted to prove that the stable integral simplicial volume is strictly bigger than the simplicial volume
for closed hyperbolic manifolds of dimension at least four. Moreover, 
we describe some possible approaches to prove or disprove that 
$\sigma_\infty (M)=\| M\|$ for every (or some) closed hyperbolic $3$-manifold $M$.

\subsection{Stable integral simplicial volume and Gromov's Question~\ref{gromov:conj}}
Proposition \ref{chi2:prop} holds (with a similar proof) also for the stable integral simplicial volume: 

\begin{prop}\label{last:prop}
Let $M$ be a closed $n$-dimensional manifold. We have
$$|\chi(M)|\leqslant (n+1)\cdot \|M\|_\infty^\matZ.$$
\end{prop}
\begin{proof}
Using Poincar\'e duality, it is not difficult to prove the following inequality~\cite{Loeh2}, 
similar to the one we used in the proof of Proposition \ref{chi2:prop}:
\begin{equation}\label{aaa}
\sum_{i=0}^n b_i(M)\leqslant (n+1)\cdot  \|M\|^\matZ
\end{equation}
(see also~\cite[Example 14.28]{luck}).
As a consequence we get $|\chi(M)|\leqslant (n+1)  \|M\|^\matZ$ and
$$
|\chi(M)|\leqslant (n+1) \cdot \|M\|^\matZ_\infty
$$ 
since the Euler characteristic is a characteristic number.
\end{proof}
Therefore if $\|M\|_\infty^\matZ$ were equal to $\|M\|$ for any aspherical manifold $M$ then we could answer positively Gromov's Question~\ref{gromov:conj}. However, as for $\sigma_\infty$, the two characteristic numbers differ at least on hyperbolic manifolds of dimension $n\geqslant 4$.
\begin{teo}\label{integral}
For every $n\geqslant 4$ there exists a constant $C_n<1$ such that the following holds.
Let $M$ be a closed orientable hyperbolic manifold of dimension $n\geqslant 4$.
Then 
$$
\|M\| \leqslant C_n \|M\|^\matZ_\infty.
$$
\end{teo}
\begin{proof}
Just as in the proof of Theorem~\ref{4:teo}, 
it is sufficient to show that 
$$
\frac{\vol(M)}{v_n}\leqslant C_n \| M\|^\matZ
$$ 
for some constant $C_n<1$ independent of $M$.
Let us now briefly describe how the proof of Theorem~\ref{bah} can be adapted to
achieve this goal.

Let $z=\sum_{i=1}^m \epsilon_i \sigma_i$ be an integral cycle representing the fundamental class of $M$, and suppose
that $\epsilon_i=\pm 1$ for every $i$ (so we don't exclude the case that $\sigma_i=\sigma_j$ for some $i\neq j$).
We may also assume that $z$ realizes the integral simplicial volume of $M$, so $\| M\|^\matZ=m$.
The cycle $z$ defines a function $f$ from the disjoint union of $m$ copies of the standard $n$-dimensional simplex
to $M$. 
Since $z$ is a cycle, there exists a complete pairing of the $(n-1)$-dimensional faces of these standard simplices
such that paired faces can be identified by a simplicial isomorphism which is compatible with $f$. These data
define a pseudomanifold $X$ endowed with a (loose) triangulation $\Tt$ with $m$ simplices and
a map $g\colon X\to M$ induced by $f$. 
By construction, the sum of the simplices of $\Tt$ defines an $n$-cycle $z'\in Z_n(X,\matZ)$, and $g_\ast(z')=z$. We can now mimic the construction described in Subsection~\ref{stmap} and homotope 
$g$ into a map $\str(g)\colon X\to M$ which sends each simplex of $\Tt$ into the support of a straight simplex
in $M$. As a consequence, we may still define positive, negative, $\varepsilon$-big and $\varepsilon$-small simplices of $\Tt$,
and full $(n-2)$-dimensional faces of $\Tt$. The estimates described in Subsection~\ref{proof:sub}
still hold, and this provides the needed constant $C_n<1$ such that $\vol(M)\leqslant C_n v_n m=C_n v_n \| M\|^\matZ$. 
\end{proof}

\subsection{$L^2$-Betti numbers}\label{futuro:sub}
As already mentioned in the introduction, Gromov was primarily interested
in the comparison of simplicial volume with the Euler characteristic. 
In~\cite{Gromov2} and~\cite{Gro3},
he suggested to use $L^2$-invariants to attack this problem.
More precisely, in \cite[page 232]{Gromov2}
he observed that Question~\ref{gromov:conj}
may be reduced to Question~\ref{gromov2:conj} below, which is formulated
in terms of $L^2$-Betti numbers.

The $L^2$-Betti numbers were first defined analytically by Atiyah in terms
of the heat kernel in the context
of cocompact groups actions on manifolds~\cite{Aty}.
Since then, the range of definition and application of $L^2$-Betti numbers  
was impressively widened, see \emph{e.g.}~\cite{Connes, CheeGro, Lupaper, Farber, Gab}.
A comprehensive introduction to $L^2$-Betti numbers may be
found in L\"uck's book \cite{luck}.

One of the most important features of $L^2$-Betti numbers is that they can be defined
both analytically and combinatorially. In the topological-combinatorial setting,
the $L^2$-invariants are mainly based on the study of the action of the fundamental
group of a space on the cellular chain complex of its universal covering. 
Here, if $M$ is a closed manifold, we denote by 
$b_k^{(2)}(M)$ the $k$-th $L^2$-Betti number 
$b_k^{(2)}(\widetilde{M},\pi_1(M))$ as defined in~\cite[Chapter 6]{luck}, where
$\widetilde{M}$ is the universal covering of $M$ and $\pi_1(M)$ acts as usual on $\widetilde{M}$.
In~\cite{Gromov2}, Gromov asked the following:

\begin{quest}[\cite{Gromov2}, page 232]\label{gromov2:conj}
Let $M$ be a closed aspherical manifold 
such that $\| M\|=0$.
Is that true that
$
b_k^{(2)}(M)=0$ 
for every
$k\in\matN
$?
\end{quest}

In order to explain how Question~\ref{gromov2:conj} is related to Question~\ref{gromov:conj},
let us briefly mention some important properties of $L^2$-Betti numbers:
\begin{enumerate}
\item
they are characteristic numbers, \emph{i.e.}~they are multiplicative with respect to finite coverings \cite[Theorem 1.35]{luck};
\item
$b_k^{(2)}(M)$ is a sort of stable
version of the
$k$-th Betti number $b_k(M)$ of $M$; more precisely, it is proved in~\cite{luckpaper} that
if $\pi_1(M)=G$ admits a sequence of nested normal subgroups $G\supseteq G_1\supseteq G_2\supseteq\ldots$
of finite index such that $\bigcap_{i\in\matN} G_i=\{1\}$, then
$$
b_k^{(2)}(M)=\inf_{\widetilde{M}_i \stackrel{d_i}{\to} M} \left\{\frac{b_k(\widetilde M_i)}{d_i}\right\},
$$
where $\widetilde M_i \stackrel{d_i}\to M$ is the covering associated to $G_i$
(see also~\cite{BerGab} for an extension of this result to the case
when $\Gamma_i$ is not supposed to be normal in $\Gamma$);
\item
if $M$ is $n$-dimensional, then $b_k^{(2)}(M)=0$ for $k>n$ and 
$$
\sum_{i=0}^n (-1)^ib_i^{(2)}(M)=\chi (M)\ 
$$
(see \cite[Theorem 1.35]{luck}).
\end{enumerate}
By property (3), a positive answer to Question~\ref{gromov2:conj} would lead to a positive answer
to Question~\ref{gromov:conj}.

\subsection{Integral foliated simplicial volume}\label{futurobis:sub}
In order to study Question~\ref{gromov2:conj}, Gromov introduced in~\cite{Gro3}
a new invariant, the \emph{integral foliated simplicial volume} $\| M\|^{\matZ,\mathcal{F}}$ of $M$ (see~\cite{Schmidt} for a precise definition and all the properties
of $\| M\|^{\matZ,\mathcal{F}}$ that we are mentioning below).
The integral foliated simplicial volume satisfies the inequalities
$$
\| M\|\leqslant \| M\|^{\matZ,\mathcal{F}}\leqslant \| M\|^\matZ .
$$
Moreover, the integral foliated simplicial volume is a characteristic number, so
$$
\| M\|\leqslant \| M\|^{\matZ,\mathcal{F}}\leqslant \| M\|^\matZ_\infty .
$$
As stated by Gromov in~\cite{Gro3} and proved by Schmidt in~\cite{Schmidt},
the integral foliated simplicial volume can be used to bound from above 
the
sum of the $L^2$-Betti numbers of a closed manifold. More precisely, if $M$
is a closed $n$-manifold then the following 
 $L^2$-analogous of inequality~\eqref{aaa}
holds~\cite{Loeh2}:
$$
\sum_{i=0}^n b_i^{(2)} (M)\leqslant (n+1) \| M\|^{\matZ,\mathcal{F}}.
$$
In particular,
a closed manifold with vanishing integral foliated simplicial volume has vanishing $L^2$-Betti numbers
(whence has vanishing Euler characteristic). However, the problem whether the 
vanishing of the simplicial volume implies the vanishing of the
integral foliated simplicial volume
at least in the case of aspherical manifolds, is still open. In fact, as far as we know, no example is known
of a closed aspherical $n$-manifold $M$ such that $\| M\|\neq \| M\|^{\matZ,\mathcal{F}}$.
Therefore, we ask here the following:
 
\begin{quest}
May our proof of Theorem~\ref{integral} be adapted to show that  $$\| M\|^{\matZ,\mathcal{F}}> \| M\|$$
for every closed hyperbolic manifold $M$ of dimension greater than 3?
\end{quest}

\subsection{The three-dimensional case}\label{futuro2:sub}
Let us now concentrate on our unsolved
\begin{quest}\label{ourquestion}
 Does the equality
$$
\| M\|=\sigma_\infty (M)
$$ 
hold for every closed hyperbolic $3$-manifold $M$?
\end{quest}

In their recent proof of the Ehrenpreis conjecture~\cite{KM}, 
Kahn and Markovic showed that every closed orientable hyperbolic surface $S$ 
has a finite covering which decomposes into pairs of pants whose boundary curves
have length arbitrarily close to an arbitrarily  big constant $R>0$.
Question~\ref{ourquestion} is equivalent to some sort of $3$-dimensional
version of Kahn and Marcovic's result. Namely, the discussion carried
out in the previous sections shows that 
$\| M\|=\sigma_\infty (M)$ if and only if the following condition holds:
for every $\varepsilon>0$, $R>0$
there exists a finite covering $\widetilde{M}$ with a triangulation $\widetilde{\Tt}$
such that  the shape of at least $(1-\varepsilon) |\widetilde{\Tt}|$ simplices of $\widetilde{\Tt}$ is $\varepsilon$-close to the shape of a regular positive simplex with edge-length 
bigger than $R$.
Let us briefly describe some possible strategies to approach Question~\ref{ourquestion}.

\subsection{How to answer Question~\ref{ourquestion} in the positive}
Corollary \ref{figure-eight:cor} says that it is possible to make the ratio $\frac{\sigma_\infty(M)}{\|M\|}$ arbitrarily close to 1 by Dehn filling the figure-eight knot complement along arbitrarily long slopes (the \emph{length} of a slope is measured with respect to a fixed toric horocusp section). This fact can be easily generalized to any manifold $M$ covering the figure-eight knot complement:

\begin{prop} Let $M_i$ be a sequence of manifolds covering the figure-eight knot. For each $i$, let $\alpha_i = (\alpha_i^1,\ldots, \alpha_i^{k_i})$ be a set of slopes on the $k_i$ boundary components of $M_i$, and let $\ell_i$ be their minimal length with respect to some horocusp section. Let $N_i$ be the closed manifold obtained by filling $M_i$ along $\alpha_i$. If $\ell_i \to \infty$ then $N_i$ is hyperbolic for all sufficiently big $i$ and we have 
$$\frac{\sigma_\infty(N_i)}{\|N_i\|} \to 1.$$
\end{prop}
\begin{proof}
Since $\ell_i$ tends to infinity, the main result of \cite{HK} ensures that $N_i$ is hyperbolic
for all sufficiently big $i$ (in fact, geometrization
and Agol's and Lackenby's results \cite{Agol, Lackenby} imply that $N_i$ is hyperbolic provided that
$\ell_i>6$).
The estimates on volume change under Dehn filling of Neumann and Zagier \cite{NZ} or Futer, Kalfagianni, and Purcell \cite[Theorem 1.1]{FKP} show that $\frac{\vol(N_i)}{\vol (M_i)}\to 1$. Let $M_i$ cover the figure-eight knot complement with degree $d_i$. Then $M_i$ is triangulated with $2d_i$ regular ideal tetrahedra and $\Vol(M_i) = 2d_iv_3$, so $\|M_i\| = \frac{\Vol(M_i)}{v_3} = 2d_i$. Therefore $\frac{\|N_i\|}{2d_i}\to 1$.

Since $M_i$ is triangulated with $2d_i$ ideal tetrahedra we have $c(M_i)\leqslant 2d_i$ (we actually have an equality, but this is not important here). Proposition \ref{filling:prop} implies that $c_\infty(N_i)\leqslant c(M_i)\leqslant 2d_i$. Since $\|N_i\|\leqslant c_\infty(N_i)$ we get $\frac{c_\infty(N_i)}{\|N_i\|} \to 1$. Finally, we have $c_\infty(N_i) = \sigma_\infty(N_i)$ because $N_i$ is irreducible.
\end{proof}

This result leads naturally to the following definition. Let $\calH_l$ be the set of all hyperbolic manifolds that may be obtained from some covering of the figure-eight knot complement by Dehn-filling some slopes having length bigger than $l$ (with respect to some horocusp section).

\begin{quest}\label{H:quest}
Let $M$ be a closed hyperbolic manifold. Is it true that for every $l$ there is a finite-degree covering $\widetilde M$ of  $M$ lying in $\calH_l$?
\end{quest}
In other words, does $M$ virtually lie in all sets $\calH_l$? A positive answer to this question would prove that $\|M\| = \sigma_\infty(M)$. Note that every manifold is a Dehn filling of some cover of the figure-eight knot, because the figure-eight knot is universal \cite{HLM}, and hence every 3-manifold lies in some $\calH_l$. 

\begin{rem}
Ehrenpreis asked \cite{Ehre} whether any two closed Riemann surfaces with $\chi <0$ have finite coverings with arbitrarily small distance, with respect to some natural metric on the moduli spaces of Riemann surfaces. This question (which was recently answered in the affirmative \cite{KM}) can be generalized to categories of manifolds of any dimension and of any kind, provided that they are equipped with a distance function: we call such a question an \emph{Ehrenpreis problem}.
The filtration $\calH_l$ easily defines a distance $d$ which translates Question \ref{H:quest} into an Ehrenpreis problem (the distance of two distinct manifolds $M$ and $N$ is the smallest $\frac 1l$ such that $M,N\in \calH_l$).
\end{rem}

We may specialize our question to the following.

\begin{quest}
Let $M$ be a closed hyperbolic manifold. Is it true that for every $l$ there is a finite-degree covering $\widetilde M$ of $M$ and a branched covering $\widetilde M \to S^3$ branched over the figure-eight knot with ramification indices bigger than $l$?
\end{quest}
In other words, does $M$ virtually cover $S^3$ with branching locus the figure-eight knot, with arbitrarily large ramification indices? A large ramification index gives a long filling slope on the covering of the figure-eight knot complement, hence a positive answer to this question would also imply a positive answer to Question \ref{H:quest} and thus 
would prove the equality $\|M\|=\sigma_\infty (M)$.

\subsection{How to answer Question~\ref{ourquestion} in the negative}
If $M$ is a closed hyperbolic $n$-manifold 
and $z_i\in Z_n(M,\matR)$ is a sequence
of representatives of the fundamental class of $M$, 
we say that $z_i$ is \emph{minimizing}
if the $L^1$-norm of $z_i$ approaches $\| M\|$ as $i$ tends to infinity. 
In order to get a negative answer to 
Question~\ref{ourquestion}, one may probably profit from Jungreis'
characterization of minimizing sequences of fundamental cycles for $M$~\cite{Jungreis}.
Every cycle $z_i$ in such a sequence lifts to a locally finite cycle $\widetilde{z}_i$
in $\matH^n$. After straightening, $\widetilde{z}_i$ is a locally finite sum of straight simplices in $\calS_n(\matH^n)$. 

Jungreis considers a suitable 
space $\mathcal{M} (\calS_n(\mathno))$
of measures on the set of geodesic simplices with vertices in $\mathno$. Every
$\widetilde{z}_i$ may be thought as a locally finite linear combination of
atomic measures concentrated on the lifts of the simplices of $z_i$. 
Therefore, 
$\widetilde{z}_i$ may be identified with an element in $\mathcal{M} (\calS_n(\mathno))$,
and Jungreis' result implies that the sequence $\widetilde{z}_i$
converges in $\mathcal{M} (\calS_n(\mathno))$ to a measure $\mu$ that is concentrated on 
the subset of regular
ideal simplices, and is invariant with respect to the action of the group
$G$ of orientation-preserving isometries of $\matH^n$. 
Roughly speaking, the $G$-invariance of $\mu$ implies that, if $n$ is big, then
the simplices of $z_i$ must be almost homogeneously distributed in $M$. Of course,
such a behaviour of $z_i$ is strongly in contrast with the possibility that 
$z_i$ is represented by a triangulation. In order to give a negative answer
to Question~\ref{ourquestion}, one could prove
that Jungreis' result is not compatible with the fact that $z_i$ is represented by a \emph{virtual}
triangulation, \emph{i.e.}~it is obtained by suitably rescaling the push-forward of a
triangulation of a finite covering.


\begin{thebibliography}{99}
\bibitem{Agol}
I. Agol, \emph{Bounds on exceptional Dehn filling}, Geom. Topol. \textbf{4} (2000),
431--449.


\bibitem{Aty}
M.~F.~Atiyah, 
\emph{Elliptic operators, discrete groups and von Neumann algebras}, 
Colloque ``Analyse et Topologie'' en l'Honneur de Henri Cartan (Orsay, 1974), pp. 43--72. Asterisque, No. 32-33, Soc. Math. France, Paris, 1976.

\bibitem{BePe}
R.~Benedetti, C.~Petronio,
\emph{Lectures on hyperbolic geometry}, Universitext. Springer-Verlag, Berlin, 1992.

\bibitem{BerGab}
N.~Bergeron, D.~Gaboriau, \emph{Asymptotique des nombres de Betti, invariants $l^2$ et laminations. [Asymptotics of Betti numbers, $l^2$-invariants and laminations]},
Comment. Math. Helv.  \textbf{79}  (2004),   362--395.

\bibitem{Bucher}
M.~Bucher-Karlsson, \emph{The proportionality constant for the simplicial volume of locally symmetric spaces},  Colloq. Math.  \textbf{111}  (2008), 183--198.

\bibitem{BFP}
M.~Bucher, R.~Frigerio, C.~Pagliantini, \emph{The simplicial volume of $3$-manifolds
with boundary}, {\tt arXiv:1208.0545}

\bibitem{CheeGro}
J.~Cheeger, M.~Gromov, 
\emph{$L_2$-cohomology and group cohomology},
Topology \textbf{25} (1986), 189--215.

\bibitem{Connes}
A.~Connes, 
\emph{Sur la th\'eorie non commutative de l'int\'egration}, Alg\`ebres d'op\'erateurs (S\'em., Les Plans-sur-Bex, 1978), pp. 19--143,
Lecture Notes in Math. \textbf{725}, Springer, Berlin, 1979.


\bibitem{Ehre}
L.~Ehrenpreis, \emph{Cohomology with bounds}, Symposia Mathematics IV, Academic Press (1970), 389--395.

\bibitem{Farber}
M.~Farber, 
\emph{von Neumann categories and extended $L^2$-cohomology},
$K$-Theory \textbf{15} (1998), 347--405.

\bibitem{FM}
K.~Fujiwara, J.~Manning, \emph{Simplicial volume and fillings of hyperbolic manifolds},
Algebr. Geom. Topol. \textbf{11} (2011), 2237--2264.

\bibitem{Francaviglia1}
S.~Francaviglia, \emph{Hyperbolic volume of representations of fundamental groups of cusped 3-manifolds},  
Int. Math. Res. Not. \textbf{9} (2004),   425--459.

\bibitem{Francaviglia2}
\bysame, \emph{Hyperbolicity equations for cusped 3-manifolds and volume-rigidity of representations}, Collana Tesi (Nuova Serie) vol 2, Edizioni Scuola Normale Superiore di Pisa, 2005. 

\bibitem{Francaviglia3}
\bysame, \emph{Similarity structures on the torus and the Klein bottle
via triangulations},
Advances in Geometry \textbf{6(3)} (2006),   397--421.

\bibitem{FriPag}
R.~Frigerio, C.~Pagliantini, \emph{The simplicial volume of hyperbolic manifolds with geodesic boundary},  
Algebr. Geom. Topol.  \textbf{10}  (2010),   979--1001.

\bibitem{FKP}
D.~Futer, E.~Kalfagianni, J.~S.~Purcell, \emph{Dehn filling, volume, and the Jones polynomial},
J. Differential Geom. \textbf{78} (2008), 429-464.

\bibitem{HM}
U.~Haagerup, H.~J.~Munkholm, 
\emph{Simplices of maximal volume in hyperbolic $n$-space}, Acta Math. 
\textbf{147} (1981), 1--11.

\bibitem{HK}
C. Hodgson, S. P. Kerckhoff,
\emph{Universal bounds for hyperbolic Dehn surgery},
Ann. Math. \textbf{162} (2005), 367--421. 

\bibitem{Gab}
D.~Gaboriau, 
\emph{Invariants $l^2$ de relations d'\'equivalence et de groupes. [$l^2$-invariants of equivalence relations and groups]},
Publ. Math. Inst. Hautes \'Etudes Sci. \textbf{95} (2002), 93--150.


\bibitem{Gro} M.~Gromov, \emph{Volume and bounded cohomology},
 Inst. Hautes \'Etudes Sci. Publ. Math. \textbf{56} (1982), 5--99.

\bibitem{Gromov2}
\bysame, \emph{Asymptotic invariants of infinite groups},
in \emph{Geometric group theory, Vol.~2 (Sussex 1991)}, 1--295, Cambridge
University Press, Cambridge, 1993.

\bibitem{Gro3} \bysame,
\emph{Metric Structures for Riemannian and non-Riemannian Spaces}
Progress in Mathematics, vol. 152, Birkh\"auser Boston Inc., Boston, MA,
1999. 

\bibitem{Ham} E.~Hamilton, \emph{Abelian subgroup separability of Haken 3-manifolds and closed
hyperbolic n-orbifolds}, Proc. London Math. Soc. \textbf{83} (2001), 626--646.

\bibitem {Hem} J.~Hempel, 
\emph{Residual finiteness for $3$-manifolds},
in ``Combinatorial group theory and topology'' (Alta, Utah, 1984),  379--396, 
Ann. of Math. Stud., vol.~111, Princeton Univ. Press, Princeton, NJ, 1987.

\bibitem{HLM} H.~M.~Hilden, M.~T.~Lozano, J.~M.~Montesinos, \emph{On knots that are universal}, 
Topology \textbf{24} (1985), 499-504. 

\bibitem{Ivanov}
N.~V.~Ivanov, \emph{Foundations of the theory of bounded cohomology},
J.~Soviet Math. \textbf{37} (1987), 1090-1114.

\bibitem{Jungreis}
D.~Jungreis,
\emph{Chains that realize the Gromov invariant of hyperbolic manifolds},
Ergodic Theory Dynam. Systems \textbf{17} (1997), 643--648.

\bibitem{KM}
J.~Kahn, V.~Markovic, \emph{The good pants homology and a proof of the Ehrenpreis conjecture},
{\tt arXiv:1101.1330}

\bibitem{Lackenby}
M. Lackenby, \emph{Word hyperbolic Dehn surgery},
Invent. Math. \textbf{140} (2000), 243--282.

\bibitem{Loeh2}
C. Loeh, \emph{Simplicial volume and $L^2$-Betti numbers},
available at www.mathematik.uni-regensburg.de/loeh/seminars/goettingen.ps


\bibitem{luckpaper} \bysame, \emph{Approximating $L^2$-invariants by their finite-dimensional analogues},
Geom. Funct. Anal. \textbf{4} (1994), 455--481.

\bibitem{Lupaper}
W.~L\"uck, \emph{Dimension theory of arbitrary modules over finite von Neumann algebras and $L^2$-Betti numbers. I. Foundations}, J. Reine Angew. Math.  \textbf{495}  (1998), 135--162.

\bibitem{luck}
W.~L\"uck, \emph{$L^2$-invariants: theory and applications to geometry and $K$-theory},
 Ergebnisse der Mathematik und ihrer Grenzgebiete. 3. Folge. A Series of Modern Surveys in Mathematics [Results in Mathematics and Related Areas. 3rd Series. A Series of Modern Surveys in Mathematics], 44. Springer-Verlag, Berlin, 2002.


\bibitem{Luo}
F.~Luo, \emph{Continuity of the volume of simplices in classical geometry}, Commun.~Cont.~Math.~\textbf{8}
(2006), 211--231.

\bibitem {Mar} B.~Martelli, \emph{Complexity of PL manifolds}, 
Algebr. Geom. Topol. \textbf{10} (2010), 1107--1164.

\bibitem {MP} B.~Martelli, C.~Petronio, \emph{Complexity of geometric three-manifolds},
Geom. Dedicata \textbf{108} (2004), 15--69.

\bibitem {Mat} S.~Matveev, \emph{Complexity theory of
    three-dimensional manifolds}, Acta Appl. Math. \textbf{19} (1990),
  101--130.
  
\bibitem {Mat:book} \bysame,  ``Algorithmic topology and classification of 3-manifolds,'' Algorithms and Computation in Mathematics, 9, Springer-Verlag, Berlin, 2003.

\bibitem {MiThu} J.~Milnor,  W.~Thurston, \emph{Characteristic numbers for three-manifolds}, Enseignment Math. \textbf{23} (1977), 249--254. 

\bibitem {NZ} W.~D.~Neumann, D.~Zagier, \emph{Volumes of hyperbolic three-manifolds}, Topology \textbf{24} (1985), 307--332. 

\bibitem{Pe}
N.~Peyerimhoff, 
\emph{Simplices of maximal volume or minimal total edge length in hyperbolic space},
J. London Math. Soc. (2) \textbf{66} (2002),  753--768.

\bibitem{Ratcliffe}
J.~G.~Ratcliffe, \emph{Foundations of hyperbolic manifolds}, Graduate Texts in
Mathematics, 149. Springer-Verlag, New York, 1994.
 
\bibitem{Schmidt}
M.~Schmidt, \emph{$L^2$-Betti numbers of $\mathcal{R}$-spaces and the
integral foliated simplicial volume}, Ph.~D.~thesis, University of M\"unster, 2005,
available at http://wwwmath.uni-muenster.de/wwwmath.uni-muenster.de/reine/u/topos/group/schmidt.pdf

\bibitem{Soma}
T.~Soma, \emph{The Gromov invariant of links},
Invent. Math. \textbf{64} (1981),  445--454.

\bibitem{Thurston}
W.~P.~Thurston, \emph{The geometry and topology of $3$-manifolds},
mimeographed notes, 1979.











\end{thebibliography}
\end{document}